\documentclass{amsart}%kh
\usepackage{amssymb}
\usepackage{amsmath}
\usepackage{latexsym}
\usepackage{amscd}
\usepackage{eufrak}
\usepackage{mathrsfs}
\usepackage[dutch,english]{babel}

\DeclareMathAlphabet{\mathpzc}{OT1}{pzc}{m}{it}
\newtheorem{theorem}{Theorem}[section]
\newtheorem{theorem-definition}[theorem]{Theorem-Definition}
\newtheorem{lemma-definition}[theorem]{Lemma-Definition}
\newtheorem{definition-prop}[theorem]{Proposition-Definition}
\newtheorem{corollary}[theorem]{Corollary}
\newtheorem{prop}[theorem]{Proposition}
\newtheorem{lemma}[theorem]{Lemma}
\newtheorem{cor}[theorem]{Corollary}
\newtheorem{definition}[theorem]{Definition}

\newtheorem{example}[theorem]{Example}

\newenvironment{remark}{\vspace{4pt}\noindent\textbf{Remark.}}{\qed\vspace{4pt}}

\newcommand{\Spf}{\ensuremath{\mathrm{Spf}\,}}
\newcommand{\Spec}{\ensuremath{\mathrm{Spec}\,}}

\newcommand{\LL}{\ensuremath{\mathbb{L}}}

\newcommand{\N}{\ensuremath{\mathbb{N}}}
\newcommand{\Z}{\ensuremath{\mathbb{Z}}}
\newcommand{\Q}{\ensuremath{\mathbb{Q}}}

\newcommand{\R}{\ensuremath{\mathbb{R}}}
\newcommand{\C}{\ensuremath{\mathbb{C}}}

\newcommand{\A}{\ensuremath{\mathbb{A}}}

\newcommand{\X}{\ensuremath{\mathscr{X}}}

\newcommand{\mY}{\ensuremath{\mathfrak{Y}}}

\renewcommand{\R}{\ensuremath{\mathbb{R}}}
\renewcommand{\C}{\ensuremath{\mathbb{C}}}

\renewcommand{\A}{\ensuremath{\mathbb{A}}}
\renewcommand{\X}{\ensuremath{\mathscr{X}}}

\renewcommand{\mY}{\ensuremath{\mathscr{Y}}}
\newcommand{\mZ}{\ensuremath{\mathfrak{Z}}}

\newcommand{\mU}{\ensuremath{\mathscr{U}}}
\newcommand{\mE}{\ensuremath{\mathfrak{E}}}
\newcommand{\mV}{\ensuremath{\mathscr{V}}}
\newcommand{\mR}{\ensuremath{\mathfrak{R}}}

\numberwithin{equation}{section}

\begin{document}
\title[Singular cohomology of the analytic Milnor fiber]{Singular cohomology of the analytic Milnor fiber,
and mixed Hodge structure on the nearby cohomology}
\author{Johannes Nicaise}
\address{Universit\'e Lille 1\\
Laboratoire Painlev\'e, CNRS - UMR 8524\\ Cit\'e Scientifique\\59655 Villeneuve d'Ascq C\'edex \\
France} \email{johannes.nicaise@math.univ-lille1.fr}
 \begin{abstract}
 We describe the homotopy type of the analytic Milnor fiber in
 terms of a strictly semi-stable model,
 and we show that its singular cohomology coincides with the
 weight zero part of the mixed Hodge structure on the nearby
 cohomology. We give a similar expression for Denef and Loeser's motivic Milnor
 fiber in terms of a strictly semi-stable model.

\vspace{5pt} \noindent MSC 2000: 32S30, 32S55, 14D07, 14G22
 %In the appendix, we consider the de Rham cohomology
 %of the analytic Milnor fiber of an isolated
 %singularity, and we re-interpret some classical results in
 %the framework of formal and rigid geometry.
 \end{abstract}

 \maketitle

\section{Introduction}
Let $k$ be any field, and let $X$ be a $k$-variety, endowed with a
morphism $f:X\rightarrow \mathrm{Spec}\,k[t]$. Let $x$ be a closed
point on the special fiber $X_s=f^{-1}(0)$. The analytic Milnor
fiber $\mathscr{F}_x$ of $f$ at $x$ was introduced and studied by
Julien Sebag and the author in \cite{NiSe-Milnor,NiSe,NiSe3}. The
object $\mathscr{F}_x$ is defined as the generic fiber of the
special formal scheme
$$\widehat{f}:\mathrm{Spf}\,\widehat{\mathcal{O}}_{X,x}\rightarrow
\mathrm{Spf}\,k[[t]]$$ It is an analytic space over the field of
Laurent series $k((t))$, and serves as a non-archimedean model of
the classical topological Milnor fibration (for $k=\C$ and $X$
smooth). The arithmetic and geometric properties of
$\mathscr{F}_x$ reflect the nature of the singularity of $f$ at
$x$. For instance, the $\ell$-adic cohomology of $\mathscr{F}_x$
is canonically isomorphic to the $\ell$-adic nearby cohomology of
$f$ at $x$ \cite[3.5]{berk-vanish2}, and the local motivic zeta
function of $f$ at $x$ can be realized as a ``Weil generating
series'' of $\mathscr{F}_x$ \cite[9.7]{Ni-trace}. If $X$ is normal
at $x$, then $\mathscr{F}_x$ is a complete invariant of the formal
germ $(f,x)$ \cite[8.8]{Ni-trace}.

In the present article, we study the topology of $\mathscr{F}_x$,
considered as a $k((t))$-analytic space in the sense of Berkovich
\cite{Berk1}. We denote by $\widehat{k((t))^a}$ the completion of
an algebraic closure of $k((t))$. Our first main result (Theorem
\ref{hom-mil}) describes the homotopy type of
$$\overline{\mathscr{F}}_x:=\mathscr{F}_x\widehat{\times}_{k((t))}\widehat{k((t))^a}$$
in terms of a so-called strictly semi-stable reduction of the germ
$(f,x)$. The second main result (Theorem \ref{main}) states that,
if $k=\C$, the singular cohomology of $\overline{\mathscr{F}}_x$
coincides with the weight zero part of the mixed Hodge structure
on the nearby cohomology of $f$ at $x$
\cite{steenbrink-vanish}\cite{navarro-invent}.

These are local analogs of results by Berkovich \cite{berk-limit}.
He showed that the weight zero part of the limit mixed Hodge
structure of a proper family $Y\rightarrow \mathrm{Spec}\,\C[t]$
coincides with the singular cohomology of the nearby fiber
$\widehat{Y}_\eta\times_{\C((t))}\widehat{\C((t))^a}$, where
$\widehat{Y}$ denotes the $t$-adic completion of $Y$ and
$\widehat{Y}_\eta$ its generic fiber.
 Our proofs closely follow the ideas in \cite{berk-limit}, but we need
 additional work to pass from the global situation to our local
 one. The most important tool in our arguments is Berkovich' description
of the homotopy type of the generic fiber of a poly-stable formal
scheme \cite{berk-contract}.

The above theorems have natural motivic counterparts. We give an
expression for Denef and Loeser's motivic Milnor fiber of $f$ at
$x$ in terms of a strictly semi-stable model (Theorem
\ref{motmil}). As we showed in \cite{Ni-trace}, this motivic
Milnor fiber can be realized as a ``motivic volume'' of the
analytic Milnor fiber $\mathscr{F}_x$. The expression in Theorem
\ref{motmil} is similar in spirit to our description of the
homotopy type of $\mathscr{F}_x$ but uses different techniques
(motivic integration) and yields another type of information
(class in the Grothendieck ring).

%Finally, in the appendix, we look at the de Rham cohomology of the
%analytic Milnor fiber of a complex isolated singularity $(f,x)$.
%We show how some classical results by Brieskorn \cite{brieskorn}
%can be interpreted in terms of formal and rigid geometry.
%Unfortunately, our main result in this context remains conjectural
%(Conjecture \ref{drham}).
\bigskip

 Let us give a survey of the structure of the paper.
In Section \ref{sec-prelim}, we recall some basic notions about
special formal schemes and their generic fibers, and about the
analytic Milnor fiber. Section \ref{sec-simp} contains the main
technical tools of the paper: we define the simplicial set
associated to a strictly semi-stable $k$-variety $X$, and we
explain how it can be used to describe the homotopy type of the
generic fiber of a strictly semi-stable formal scheme $\X$,
following results by Berkovich. Moreover, we prove a crucial
result about the homotopy type of the generic fiber of the
completion of $X$ along a union of irreducible components
(Proposition \ref{deform} and its Corollary \ref{ssvar}). It is
this property that allows to pass from Berkovich' global situation
to our local one.

In Section \ref{sec-hom}, these tools are applied to establish
homotopy-equivalences between certain analytic spaces, and to
study the homotopy type of the analytic Milnor fiber (Theorem
\ref{hom-mil}). The key result is Proposition \ref{homeq}, which
we now briefly explain. Any special formal $k[[t]]$-scheme $\X$
can be considered as a special formal $k$-scheme $\X^k$, by
forgetting the $k[[t]]$-structure. The generic fiber of $\X^k$
(where $k$ carries the trivial absolute value) is naturally
fibered over $[0,1[$ by evaluating the points of the generic fiber
(which are multiplicative semi-norms) in $t$. Proposition
\ref{homeq} studies the homotopy type of the fibers of this
family.

%Section \ref{sec-arc} is an intermezzo explaining the relation
%between the arc space of a $k$-variety, and its associated
%$k$-analytic space (where $k$ is a field endowed with the trivial
%absolute value). The results in this Section are not used in the
%remainder of the article.

 Section \ref{sec-mhs} contains the results concerning the
mixed Hodge structure on the nearby cohomology of $f$ at $x$
(where $k=\C$). We recall Peters and Steenbrink's construction of
this mixed Hodge structure, and we show how the weight zero part
can be computed on a strictly semi-stable reduction (Proposition
\ref{mhs}). Combining this result with the ones in Section
\ref{sec-hom}, we see that the singular cohomology of the analytic
Milnor fiber coincides with the weight zero part of the mixed
Hodge structure on the nearby cohomology of $f$ at $x$ (Theorem
\ref{main}).

In the final Section \ref{sec-motmil}, we compare the preceding
results with the motivic setting, and give an expression for Denef
and Loeser's motivic nearby cycles and motivic Milnor fiber in
terms of a strictly semi-stable model (Theorem \ref{motmil} and
Corollary \ref{cor-motmil}).
%Finally, in the Appendix, we consider the de Rham cohomology of
%the analytic Milnor fiber of a complex isolated singularity
%$(f,x)$. We re-interpret some classical results by Brieskorn in
%our setting, and we give a conjectural description of the de Rham
%cohomology of $\mathscr{F}_x$.
\section{Preliminaries}\label{sec-prelim}
\subsection{Some notation}
If $S$ is any scheme, a $S$-variety is a separated reduced scheme
of finite type over $S$.

For any locally ringed space $Y$, we denote by $|Y|$ its
underlying topological space; if no risk of confusion arises,
we'll omit the vertical bars and we'll denote the underlying
topological space by $Y$, to simplify notation.

For any field $F$, we denote by $F^{a}$ an algebraic closure, and
by $F^s$ the separable closure of $F$ in $F^a$. If $L$ is a
non-archimedean field (we do not exclude the trivial valuation),
then the absolute value on $L$ extends uniquely to an absolute
value on $L^a$, and we denote by $\widehat{L^s}$ and
$\widehat{L^{a}}$
 the completions of $L^s$, resp. $L^a$ (these completions coincide if the valuation is non-trivial).
  We will
work in the category of $L$-analytic spaces as introduced by
Berkovich \cite{Berk-etale}. For any $L$-analytic space $X$, we
put $\overline{X}=X\widehat{\times}_L \widehat{L^a}$.

If $R$ is a discrete valuation ring, with residue field $k$, and
$Y$ is a scheme over $R$, then we denote its special fiber
$Y\times_R k$ by $Y_s$.

If $T$ is any topological space and $S$ is a subspace of $T$, then
a strong deformation retract of $T$ onto $S$ is a continuous map
$\phi:T\times[0,1]\rightarrow T$ such that $\phi(\cdot,0)$ is the
identity, $\phi(s,r)=s$ for any point $s$ of $S$ and any $r\in
[0,1]$, and $\phi(\cdot,1)$ is a surjection onto $S$.

\subsection{Strictly semi-stable (formal)
schemes}\label{subsec-strictsemi} Let $K$ be a complete discretely
valued field, with valuation $v_K:K^*\rightarrow \Z$. We do not
assume that the valuation is non-trivial, but we will assume that
it is normalized (i.e. $v_K(K^*)=\Z$ or $v_K(K^*)=\{0\}$). We
denote by $K^o$ the ring of integers in $K$, and by
$\widetilde{K}$ the residue field. If $v_K$ is trivial, then
$K=K^o=\widetilde{K}$. For each $r\in \,]0,1[$, we define an
absolute value $|\cdot|_{r}$ on $K$ by $|x|_{r}=r^{v_K(x)}$ and we
denote by $K_r=(K,|\cdot|_{r})$ the corresponding non-archimedean
field. Moreover, we put $K_0=(\widetilde{K},|\cdot|_{0})$ where
$|\cdot|_{0}$ is the trivial absolute value. The distinction
between $K$ (which only carries a discrete valuation) and the
fields $K_r$ (which are non-archimedean fields) is crucial for the
applications in this article.

If $L=(L,|\cdot|_L)$ is any non-archimedean field, then we denote
by $L^o$ the ring of integers and by $\widetilde{L}$ the residue
field. In particular, $K_r^o=K^o$ and
$\widetilde{K_r}=\widetilde{K}$ for $r\in \,]0,1[$.
 A formal scheme $\X$
over $L^o$ is called $stft$ if it is separated, and topologically
of finite type over $L^o$. We denote by $\X_\eta$ its generic
fiber, a $L$-analytic space, and by $\X_s$ its special fiber, a
separated scheme of finite type over $\widetilde{L}$. There is a
natural specialization map (of sets\footnote{In fact, if
$|\cdot|_L$ is non-trivial, $sp_{\X}$ can be enhanced to a
morphism of ringed sites: the specialization morphism
$sp_{\X}:\X_\eta\rightarrow \X$, where $\X_\eta$ is endowed with
the strong $G$-topology; note that $|\X|=|\X_s|$.})
$sp_{\X}:\X_\eta\rightarrow \X_s$. If $X$ is a separated scheme of
finite type over $L$, we denote by $X^{an}$ the $L$-analytic space
associated to $X$ via non-archimedean GAGA \cite[3.4-5]{Berk1}.

A special formal $K^o$-scheme is a separated adic Noetherian
formal scheme $\X$, endowed with a morphism $\X\rightarrow
\mathrm{Spf}\,K^o$, such that $\X/\mathcal{J}$ is of finite type
over $K^o$ for any ideal of definition $\mathcal{J}$ on $\X$ (this
definition is slightly more restrictive then the one used by
Berkovich in \cite{berk-vanish2}). In particular, any $stft$
formal $K^o$-scheme is special. We denote by $(SpF/K^o)$ the
category of special formal $K^o$-schemes.

There is a functor
$$(\cdot)_0:(SpF/K^o)\rightarrow (sft/\widetilde{K}):\X\mapsto \X_0$$
to the category $(sft/\widetilde{K})$ of separated
$\widetilde{K}$-schemes of finite type, where $\X_0$ is the
\textit{reduction} of $\X$ (the closed subscheme of $\X$ defined
by the largest ideal of definition in $\mathcal{O}_{\X}$).

Moreover, we consider the functor
$$(\cdot)_s:(SpF/K^o)\rightarrow (SpF/\widetilde{K}):\X\mapsto \X_s=\X\times_{\mathrm{Spf}\,K^o}
\mathrm{Spec}\,\widetilde{K}$$
%to the category of special formal
%$\widetilde{K}$-schemes,
mapping $\X$ to its special fiber $\X_s$.
If $\X$ is $stft$ over $K^o$, then $\X_s$ is a scheme, and $\X_0$
is the maximal reduced closed subscheme of $\X_s$; in any case,
$(\X_s)_0\cong \X_0$.

 For any
special formal $K^o$-scheme $\X$ and any value $r\in \,]0,1[$, we
denote by $\X(r)$ the formal scheme $\X$ viewed as a special
formal $K_r^o$-scheme via the identification $K^o=K_r^o$. We
denote by $\X(r)_\eta$ the generic fiber of $\X(r)$ in the
category of $K_r$-analytic spaces. We also put $\X(0)=\X_s$. It is
a special formal scheme over $\widetilde{K}$, and we denote its
generic fiber in the category of $K_0$-analytic spaces by
$\X(0)_\eta$. For each $r\in[0,1[$, there is a canonical
specialization map (of sets) $sp_{\X(r)}:\X(r)_\eta \rightarrow
\X_0$.

If $X$ is a separated $\widetilde{K}$-scheme of finite type, then
we can view $X$ also as a $stft$ formal $\widetilde{K}$-scheme,
and there exists a canonical morphism of $K_0$-analytic spaces
$X_\eta\rightarrow X^{an}$, which is an isomorphism iff $X$ is
proper \cite[1.10]{thuillier}. Hence, if $\X$ is a proper $stft$
formal $R$-scheme, then $\X(0)_\eta$ is canonically isomorphic to
$\X_s^{an}$.

 A $stft$ formal $K^o$-scheme is called
strictly semi-stable if it can be covered with affine open formal
subschemes $\mU$, endowed with an \'etale morphism of formal
$K^o$-schemes of the form
$$\mU\rightarrow \mathrm{Spf}\,K^o\{x_0,\ldots,x_m\}/(x_0\cdot x_1\cdot \ldots\cdot x_p-\pi)$$
for some $p\leq m$, where $\pi$ generates the maximal ideal of
$K^o$ (in particular, $\pi=0$ if the valuation on $K$ is trivial).

This definition includes as a special case the class of strictly
semi-stable $\widetilde{K}$-varieties,
%(where $K$ carries the
%trivial valuation),
i.e. varieties which admit Zariski-locally an \'etale map to a
$\widetilde{K}$-scheme of the form
$\mathrm{Spec}\,\widetilde{K}[x_0,\ldots,x_m]/(x_0\cdot x_1\cdot
\ldots\cdot x_p)$. If $\X$ is a strictly semi-stable formal
$K^o$-scheme, then $\X_s$ is a strictly semi-stable
$\widetilde{K}$-variety.

If $X$ is a strictly semi-stable $\widetilde{K}$-variety, we
denote by $Irr(X)$ its set of irreducible components. For any
non-empty subset $J$ of $Irr(X)$, we put $X_J=\cap_{V\in J}V$ and
$X_J^o=X_J\setminus \cup_{W\in (Irr(X)\setminus J)}W$.

\subsection{The analytic Milnor fiber}\label{subsec-anmil}
 Let $k$ be any field, and put $K=k((t))$, endowed with the $t$-adic
 valuation. Fix a value $r\in \,]0,1[$.
 Let $X$ be a variety over $k$, endowed
with a morphism of $k$-schemes $f:X\rightarrow
\mathrm{Spec}\,k[t]$. The $t$-adic completion of $f$ is a $stft$
formal $K^o$-scheme $\X$, whose special fiber $\X_s$ is
canonically isomorphic to the fiber of $f$ over the origin. For
any closed point $x$ of $\X_s$, the set $sp^{-1}_{\X(r)}(x)$ is
open in $\X(r)_\eta$ and inherits the structure of a
$K_r$-analytic space. By \cite[0.2.7]{bert} it is canonically
isomorphic to the generic fiber of the special formal
$K_r^o$-scheme $\mathrm{Spf}\,\widehat{\mathcal{O}}_{X,x}$, where
the $K_r^o$-structure is defined by $f$. In \cite{NiSe}, we called
$sp^{-1}_{\X(r)}(x)$ the \textit{analytic Milnor fiber} of $f$ at
$x$, and we denoted it by $\mathscr{F}_x$. By
\cite[3.5]{berk-vanish2}, if $k$ is separably closed, the
$\ell$-adic cohomology of $\mathscr{F}_x$ is canonically
isomorphic to
 the cohomology of the stalk of the complex
of $\ell$-adic nearby cycles at $x$ (here $\ell$ is a prime
invertible in $k$): for each integer $i\geq 0$, there is a
canonical $G(K^s/K)$-equivariant isomorphism
$$H^i_{\acute{e}t}(\mathscr{F}_x\widehat{\times}_{K_r}\widehat{K_r^s},\Q_\ell)\cong R^i\psi_\eta(\Q_\ell)_x$$
%We call $X$ \textit{strongly} semi-stable if
%$X_J$ is irreducible for each $J$. More generally, a strictly
%semi-stable formal $K^o$-scheme $\X$ is called strongly
%semi-stable if its special fiber $\X_s$ has this property. The
%following lemma is trivial.
%
%\begin{lemma}
%Suppose that the valuation on $K$ is non-trivial. If $\X$ is a
%strictly semi-stable formal $K^o$-scheme, then there exists an
%admissible blow-up $\X'\rightarrow \X$ such that $\X'$ is strongly
%semi-stable.
%\end{lemma}
%
%More precisely, we can construct such a $\X'$ by blowing-up
%irreducible components of the spaces $(\X_s)_J$.

\section{The simplicial set associated to a strictly
semi-stable variety}\label{sec-simp}
\subsection{Definitions}\label{simpint}
In this section, we define the simplicial complex $\Delta(X)$
associated to a strictly semi-stable $\widetilde{K}$-variety $X$,
and we list some basic properties.

\begin{remark}
In
 \cite[\S\,4.1]{thuillier}, Thuillier defines $\Delta(X)$ as an oriented simplicial
 complex by choosing a total order on the set of irreducible components of
 $X$. It seems more natural to construct $\Delta(X)$ as an
 unoriented simplicial set, independent of all choices. This is
 the approach we adopt here. The result is isomorphic
 to the unoriented simplicial set underlying Thuillier's
 construction.
 \end{remark}

For any pair of categories $\mathcal{C},\,\mathcal{D}$, we denote
by $\mathcal{D}^o\mathcal{C}$ the category of presheaves on
$\mathcal{D}$ with values in $\mathcal{C}$. Its objects are
functors $\mathcal{D}^{o}\rightarrow \mathcal{C}$, where
$\mathcal{D}^o$ denotes the opposite category of $\mathcal{D}$,
and its morphisms are natural transformations of functors. We
denote by $(Ens)$ the category of sets.

For any integer $n\geq 0$, we denote by $[n]$ the set
$\{0,\ldots,n\}$, and we define a category $\Delta$ whose objects
are the sets $[n]$ for $n\geq 0$, and whose arrows are maps of
 sets. The category
$\Delta^o(Ens)$ is called the category of (unoriented) simplicial
sets.

The object of $\Delta^o(Ens)$ represented by $[n]$ is called the
standard $n$-simplex, and denoted by $\Delta[n]$. If $\Sigma$ is a
simplicial set and $n\geq 0$ is an integer, we call $\Sigma([n])$
the set of $n$-simplices of $\Sigma$. Note that there exists a
natural bijection $\Sigma([n])\cong
Hom_{\Delta^o(Ens)}(\Delta[n],\Sigma)$. A $n$-simplex $\gamma$ is
called degenerate if there exists a map of sets $g:[n]\rightarrow
[n-1]$ such that $\gamma$ belongs to the image of
$\Sigma(g):\Sigma([n-1])\rightarrow \Sigma([n])$. A $m$-simplex
$\gamma_1$ and a $n$-simplex  $\gamma_2$ are called equivalent if
there exist maps $[m]\rightarrow [n]$ and $[n]\rightarrow [m]$
mapping $\gamma_2$ to $\gamma_1$, resp. $\gamma_1$ to $\gamma_2$.
For any set $I$, we define the associated simplicial set
$\Delta_I$ by $$\Delta_I([n])=Hom_{(Ens)}([n],I)$$ for all $n\geq
0$.

There is a natural geometric realization functor
$$|\cdot|:\Delta^o(Ens)\rightarrow (Ke):\Sigma \mapsto |\Sigma|$$
where $(Ke)$ denotes the category of Kelley spaces (topological
Hausdorff spaces $T$ such that a subset of $T$ is closed iff its
intersection with all compact subsets of $T$ is closed). This
functor is characterized by the fact that it commutes with direct
limits, maps the standard $n$-simplex $\Delta[n]$ to the
topological $n$-simplex
$$\Delta_n=\{(u_0,\ldots,u_n)\in
[0,1]^{n+1}\,|\,\sum_{i=0}^{n}u_i=1\}$$ and sends a map of sets
$\alpha:[m]\rightarrow [n]$ to the unique affine map
$|\alpha|:\Delta_m\rightarrow \Delta_n$ sending the vertex $v_i$
of $\Delta_m$ to the vertex $v_{\alpha(i)}$ of $\Delta_n$ (for
$n\geq 0$ and $i\in [n]$, the vertex $v_i$ of $\Delta_n$ is the
point with coordinates $u_j=\delta_{ij}, \,j=0,\ldots,n$).

If $\Sigma$ is a simplicial set, then a $n$-cell of $|\Sigma|$ is
the image of the interior $(\Delta_n)^o$ under the map
$|\gamma|:\Delta_n\rightarrow |\Sigma|$ induced by some
non-degenerate $n$-simplex $\gamma\in \Sigma([n])$. Two simplices
$\gamma_1\in \Sigma([m])$ and $\gamma_2\in \Sigma([n])$ are
equivalent, iff the images of the corresponding maps
$|\gamma_1|:\Delta_m\rightarrow |\Sigma|$ and
$|\gamma_2|:\Delta_n\rightarrow |\Sigma|$ coincide.
%morphism $h:\Delta[n]\rightarrow \Sigma$
%in $\Delta^o(Ens)$.

Now let $X$ be a strictly semi-stable $\widetilde{K}$-variety. We
denote by $Str(X)$ the set consisting of the generic points of the
closed subsets $X_J$ of $X$, where $J$ varies over the non-empty
subsets of $Irr(X)$. In particular, $Str(X_J)\subset Str(X)$ is
the set of generic points of the smooth variety $X_J$. Identifying
an element of $Irr(X)$ with its generic point, we can consider
$Irr(X)$ as a subset of $Str(X)$ in a natural way.

For any point $x$ of $X$, we denote by $\Psi(x)$ the set of
elements of $Irr(X)$ containing $x$~; if we want to make $X$
explicit, we write $\Psi_X(x)$ instead of $\Psi(x)$.
%We will often identify
%$\Psi(x)$ with the set of indices $i\in I$ such that $\Psi(x)$
%contains the generic point of $X_i$, and
For any point $x$ of $Str(X)$, we denote by $S_x$ the connected
component of $X_{\Psi(x)}^o$ containing $x$. Then $x$ is the
generic point of $S_x$, and $\{S_x\,|\,x\in Str(X)\}$ is a finite
stratification of $X$ into smooth irreducible locally closed
subsets. We call any union of strata a strata subset of $X$.

We define a partial ordering on $Str(X)$ as follows: $x\leq y$ iff
$y$ belongs to the Zariski closure of $\{x\}$ in $X$. For each
$x\in Str(X)$ we consider the simplicial set
$\Delta(x)=\Delta_{\Psi(x)}$. For $x\leq y$, we have
$\Psi(x)\subset \Psi(y)$, and hence a natural morphism of
simplicial sets $\Delta(x)\rightarrow \Delta(y)$. This defines a
functor
$$\Delta(\cdot):Str(X)\rightarrow \Delta^o(Ens)$$ and we define
$\Delta(X)$ as its colimit:
$$\Delta(X)=\lim_{\stackrel{\longrightarrow}{Str(X)}}\Delta(\cdot)$$

We can give a more explicit construction of $\Delta(X)$ as
follows. Denote, for each $n\geq 0$, by $D(X)([n])$ the set of
couples $(x,f)$ with $x\in Str(X)$, and $f$ a surjection
$[n]\rightarrow \Psi(x)$. For any map $\alpha:[m]\rightarrow [n]$,
we define a map
$$D(X)(\alpha):D(X)([n])\rightarrow D(X)([m]):(x,f)\mapsto (x',f')$$
where $f':=f\circ \alpha$, and $x'$ is the unique point of
$Str(X)$ such that $\Psi(x')=Im(f')$ and $x$ belongs to the
Zariski closure of $\{x'\}$ in $X$. In this way, $D(X)$ becomes a
simplicial set. The following lemma can be verified in a
straightforward way.

\begin{lemma}\label{simpset}
 There exists a canonical
isomorphism of simplicial sets $\Delta(X)\cong D(X)$. In
particular, for each $n\geq 0$, there is a canonical bijection
between the set of equivalence classes of non-degenerate
$n$-simplices (or, equivalently, the set of $n$-cells) of
$\Delta(X)$ and the set $\cup_{J\subset Irr(X), |J|=n+1}Str(X_J)$.
\end{lemma}

Hence, the non-degenerate $n$-simplices $\gamma$ of $\Delta(X)$
correspond to couples $(x,f)$ with $x\in Str(X)$ and $f$ a
bijection $[n]\cong \Psi(X)$. A point $z$ of the corresponding
cell of $|\Delta(X)|$ is the image of a unique point
$(u_0,\ldots,u_n)$ of $(\Delta_n)^o$ under the map
$|\gamma|:\Delta_n\rightarrow |\Delta(X)|$, and we define a tuple
$v\in \, ]0,1]^{\Psi(x)}$ by $v(i)=u_{f^{-1}(i)}$ for $i\in
\Psi(x)$. This tuple is invariant under equivalence of
non-degenerate $n$-simplices, and the point $z$ of $|\Delta(X)|$
is completely determined by the couple $(x,v)$.

\begin{definition}[Barycentric representation]\label{bary}
We call $v$ the tuple of barycentric coordinates of the point $z$,
and we call $(x,v)$ the barycentric representation of $z$.
\end{definition}

By Lemma \ref{simpset}, there is a natural bijection $x\mapsto
C_x$ from $Str(X)$ to the set of cells of $|\Delta(X)|$. Whenever
$E$ is a strata subset of $X$, we denote by $|\Delta_E(X)|$ the
union of the cells $C_x$ with $x\in Str(X)\cap E$.

\subsection{Comparison with Berkovich' definition}
Denote by $\widetilde{\Delta}$ the subcategory of $\Delta$ with
the same objects but with only injective maps. This is a full
subcategory of the category $\Lambda$ considered in
\cite[p.\,24]{berk-contract}. If $X$ is a strictly semi-stable
$\widetilde{K}$-variety, then Berkovich defines in
\cite[p.\,29]{berk-contract} an object $C(X)$ of the category
$\Lambda^o(Ens)$ of presheaves on $\Lambda$, and its geometric
realization $|C(X)|$. The aim of this section is to compare these
objects with the simplicial set $\Delta(X)$ defined above, and its
geometric realization $|\Delta(X)|$.

A simplicial set $\Sigma$ is called non-degenerate, if for any
injective map $[m]\rightarrow [n]$ in $\Delta$, the induced map
$\Sigma([n])\rightarrow \Sigma([m])$ takes non-degenerate
$n$-simplices to non-degenerate $m$-simplices. A morphism of
simplicial sets is called non-degenerate if it takes
non-degenerate simplices to non-degenerate ones. We denote the
subcategory of $\Delta^o(Ens)$ consisting of non-degenerate
simplicial sets with non-degenerate morphisms between them by
$\Delta^o(Ens)_{nd}$.

The embedding of $\widetilde{\Delta}$ in $\Lambda$, resp. $\Delta$
induces forgetful functors $F:\Lambda^o(Ens)\rightarrow
(\widetilde{\Delta})^o(Ens)$ and $G:\Delta^o(Ens)\rightarrow
(\widetilde{\Delta})^o(Ens)$.

\begin{lemma}\label{adjoint}
The functor $G:\Delta^o(Ens)\rightarrow
(\widetilde{\Delta})^o(Ens)$ has a left adjoint
$H:(\widetilde{\Delta})^o(Ens)\rightarrow \Delta^o(Ens)$. This
functor is a faithful embedding, its image is contained in
$\Delta^o(Ens)_{nd}$, and
$H:(\widetilde{\Delta})^o(Ens)\rightarrow\Delta^o(Ens)_{nd}$ is an
equivalence of categories.
%Moreover, for any object $\Sigma$ of
%$\Delta^o(Ens)_{nd}$, the natural map $(H\circ
%G)(\Sigma)\rightarrow \Sigma$ is an isomorphism.
\end{lemma}
\begin{proof}
For any integer $n\geq 0$, we denote by $\widetilde{\Delta}[n]$
the object of $(\widetilde{\Delta})^o(Ens)$ represented by $[n]$.
For any object $S$ of $(\widetilde{\Delta})^o(Ens)$, we denote by
$\widetilde{\Delta}/S$ the category of morphisms
$\alpha:\widetilde{\Delta}[n]\rightarrow S$, with $n\geq 0$, where
a morphism in $\widetilde{\Delta}/S$ from $\alpha$ to
$\beta:\widetilde{\Delta}[m]\rightarrow S$ is a map
$f:\widetilde{\Delta}[n]\rightarrow \widetilde{\Delta}[m]$ in
$(\widetilde{\Delta})^o(Ens)$ with $\alpha=\beta\circ f$. We
define the functor $H$ by putting $H_S(\alpha)=\Delta[n]$ for any
object $\alpha:\widetilde{\Delta}[n]\rightarrow S$ of
$\widetilde{\Delta}/S$, and
$$H(S)=\lim_{\stackrel{\longrightarrow}{
\widetilde{\Delta}/S}} H_S(\cdot)$$ In particular,
$H(\widetilde{\Delta}[n])=\Delta[n]$. The action of $H$ on
morphisms in $(\widetilde{\Delta})^o(Ens)$ is the obvious one.

Now we check that $H$ is indeed a left adjoint for $G$. Let
$\Sigma$ be any object of $\Delta^o(Ens)$, and let $S$ be any
object of $(\widetilde{\Delta})^o(Ens)$. By definition,
$$Hom_{\Delta^o(Ens)}(H(S),
\Sigma)=\lim_{\stackrel{\longrightarrow}{\widetilde{\Delta}[n]\rightarrow
S}}\Sigma([n])$$ and since
$$\Sigma([n])=G(\Sigma)([n])=Hom_{(\widetilde{\Delta})^o(Ens)}(\widetilde{\Delta}[n],G(\Sigma))$$
it suffices to note that
$$S\cong \lim_{\stackrel{\longrightarrow}{\widetilde{\Delta}[n]\rightarrow S}} \widetilde{\Delta}[n]
$$for any object $S$ of $(\widetilde{\Delta})^o(Ens)$, by
\cite[II.1.1]{Gabriel}.

For any object $S$ of $(\widetilde{\Delta})^o(Ens)$, and each
integer $n\geq 0$, we consider the set $\mathcal{S}_n(S)$
consisting of triples $(p,f,\gamma)$ where $p\geq 0$, $f$ is a
surjection $[n]\rightarrow [p]$, and $\gamma$ is an element of
$S([p])$. We define an equivalence relation on $\mathcal{S}_n(S)$
by stipulating that $(p,f,\gamma)\sim(p',f',\gamma')$ iff $p=p'$
and there exists an automorphism $\varphi$ of $[p]$ such that
$f'=\varphi\circ f$ and $\gamma=S(\varphi)(\gamma')$.

 It is not
hard to see that, for each integer $n\geq 0$, there exists a
canonical bijection between the set of $n$-simplices $H(S)([n])$
and the quotient set $\mathcal{S}_n(S)/\sim$.
  The $n$-simplex represented by $(p,f,\gamma)$ is
non-degenerate iff $p=n$. If $\alpha:[m]\rightarrow [n]$ is a
morphism in $\Delta$, then
$$H(S)(\alpha):H(S)([n])\rightarrow H(S)([m])$$ maps $(p,f,\gamma)$ to
$$(q,(f\circ \alpha:[m]\rightarrow Im(f\circ \alpha)\cong [q]),
\gamma')$$ where we chose an isomorphism $[q]\cong Im(f\circ
\alpha)$ and $\gamma'$ is the image of $\gamma$ in $H(S)([q])$
w.r.t. the inclusion map $[q]\cong Im(f\circ \alpha)\rightarrow
[p]$. Using this description, one can see that for any morphism
$h:S\rightarrow T$ in $(\widetilde{\Delta})^o(Ens)$, the image
$H(h):H(S)\rightarrow H(T)$ is a non-degenerate morphism between
non-degenerate simplicial sets.
%It also makes clear that $H$ is a
%faithful embedding: for each $n\geq 0$, there is a canonical
%bijection between $S([n])$ and the non-degenerate $n$-simplices of
%$H(S)$, and we recover $h:S\rightarrow D$ from
%$H(h):H(S)\rightarrow H(D)$ by restricting to the non-degenerate
%simplices.

To conclude, we construct a quasi-inverse $I$ for
$H:(\widetilde{\Delta})^o(Ens)\rightarrow \Delta^o(Ens)_{nd}$. For
any non-degenerate simplicial set $\Sigma$ and any $n\geq 0$, we
define $I(\Sigma)([n])$ as the set of non-degenerate $n$-simplices
of $\Sigma$. By the definitions of non-degenerate simplicial set
and non-degenerate morphism, we can make $I$ into a functor in the
obvious way. It is easy to see that $I$ is quasi-inverse to $H$.
%it suffices to prove that, for any non-degenerate
%morphism of non-degenerate simplicial sets $g:\Sigma\rightarrow
%\Upsilon$, the natural diagram
%$$\begin{CD}
%(H\circ G)(\Sigma) @>>> \Sigma
%\\ @V(H\circ G)(g)VV @VVgV
%\\ (H\circ G)(\Upsilon) @>>> \Upsilon
%\end{CD}$$
%commutes, and that the horizontal maps (obtained by adjunction)
%are isomorphisms. Commutativity is an immediate consequence of the
%adjunction property.
%%(since any functor from a small category to the category of sets
%%is canonically a direct limit of representable functors).
\end{proof}

\begin{prop}\label{compar}
For any strictly semi-stable $\widetilde{K}$-variety $X$, there
exists a canonical isomorphism of simplicial sets $\alpha:(H\circ
F)(C(X))\rightarrow \Delta(X)$. In particular, $\Delta(X)$ is
non-degenerate. Moreover, there exists a canonical homeomorphism
$|\Delta(X)|\rightarrow |C(X)|$.
\end{prop}
\begin{proof}
Recall that, for any point $x$ of $Str(X)$, we denote by $\Psi(x)$
the set of irreducible components of $X$ containing $x$. Then, by
definition, for any integer $n\geq 0$, $F(C(X))([n])$ is the set
of pairs $(x,\alpha)$ with $x\in Str(X)$ and $\alpha$ a bijection
$[n]\rightarrow \Psi(x)$. By the construction of the functor $H$
in the proof of Lemma \ref{adjoint}, we see that $(H\circ
F)(C(X))([n])$ is the set of pairs $(x,\beta)$ with $x\in Str(X)$
and $\beta$ a surjection $[n]\rightarrow \Psi(x)$. Now the
existence of the canonical isomorphism  follows from Lemma
\ref{simpset}. The existence of a canonical homeomorphism
$|(H\circ F)(C(X))|\rightarrow |C(X)|$ is clear from Berkovich'
construction of $|C(X)|$.
\end{proof}

\subsection{The skeleton of a strictly semi-stable formal
scheme}\label{berk} Let $(L,|\cdot|_L)$ be an arbitrary
non-archimedean field, with ring of integers $L^o$ and residue
field $\widetilde{L}$ (we do not exclude the trivial absolute
value). We recall a particular case of Berkovich' definition of
the skeleton $S(\X)$ of a poly-stable formal $L^o$-scheme $\X$,
and his construction of a strong deformation retract of $\X_\eta$
onto $S(\X)$ (see \cite[\S\,5]{berk-contract}). We establish some
basic properties for use in the following sections.

\textit{Case 1.} Suppose that $\X=\mathrm{Spf}\,A$ with
$A=L^o\{x_0,\ldots,x_m\}/(x_0\cdot \ldots\cdot x_p-\alpha)$ for
some $p\leq m$, and with $\alpha\in L^o$, $|\alpha|_L<1$. Each
element of $A_L:=A\otimes_{L^o}L$ has a unique representant in the
set
$$D=\{\sum_{i\in \N^{[m]}}a_i x^i\in L\{x_0,\ldots,x_m\}\,|\,a_i=0\mbox{ if
}\min\{i_0,\ldots,i_p\}>0\}$$  i.e. the natural map $D\rightarrow
A_L$ is a bijection (and even an isometry if we endow $A_L$ with
the quotient norm w.r.t. the given presentation
$L\{x_0,\ldots,x_m\}\rightarrow A_L$). Put
$$S=\{r\in [0,1]^{[p]}\,|\,r_0\cdot \ldots\cdot r_p=|\alpha|_L\}$$
and consider the map $\theta:S\rightarrow \X_\eta$ mapping $r$ to
the element $\theta(r)$ of $\X_\eta=\mathscr{M}(A_L)$ defined by
$$\theta(r):D\rightarrow \R_+:\sum_{i}a_ix^i\mapsto \max_{i\in
\N^{[m]}}\{|a_i|_Lr^i\}$$ where $r^i=r_0^{i_0}.\ldots.r_p^{i_p}$
with the convention that $0^0=1$. The map $\theta$ is a
homeomorphism onto its image $\theta(S)$, which is by definition
the skeleton $S(\X)$ of $\X$, and $\theta$ is right inverse to the
map $\phi:\X_\eta\rightarrow S$ mapping a point $z$ of $\X_\eta$
to the tuple $(|x_0(z)|,\ldots,|x_p(z)|)$. If we put
$\tau_{\X}=\theta\circ \phi$, then Berkovich constructed a natural
strong deformation retract
$\Phi_{\X}:\X_\eta\times[0,1]\rightarrow \X_\eta$ with
$\Phi_{\X}(\cdot,1)=\tau_{\X}$.

\textit{Case 2.} Now we consider the case where $\X$ admits an
\'etale map $h:\X\rightarrow \mY=\mathrm{Spf}\,A$, with $A$ as
above. Then the skeleton $S(\X)$ is the inverse image of $S(\mY)$
under $h_\eta$, and it does not depend on the choice of the map
$h$. % (see \cite[\S\,4.2]{berk-contractII}).
Moreover, Berkovich
describes the map $\Phi_{\mY}$ in terms of a certain torus action,
which lifts uniquely to $\X_\eta$; in this way, he defines a
natural strong deformation retract
$\Phi_{\X}:\X_\eta\times[0,1]\rightarrow \X_\eta$ of $\X_\eta$
onto $S(\X)$, such that $\Phi_{\mY}(h_\eta(x),\rho)=h_\eta\circ
\Phi_{\X}(x,\rho)$ for each point $x$ of $\X_\eta$ and each
$\rho\in [0,1]$. This map $\Phi_{\X}$ does not depend on $h$.

\textit{Case 3.} Finally, if $\X$ is a $stft$ formal $L^o$-scheme
that can be covered by open formal subschemes with the property of
Case 2, then the construction of the skeleton $S(\X)$ and the
strong deformation retract $\Phi_{\X}$ are obtained by gluing the
constructions in the previous step. We put
$\tau_{\X}=\Phi_{\X}(\cdot,1):\X_\eta\rightarrow S(\X)$. In
particular, Case 3 applies to all strictly semi-stable formal
$K^o$-schemes.

The map $\Phi_{\X}$ has the following property: for any point $x$
of $\X_\eta$, we have $(sp_{\X}\circ\Phi_{\X})(x,\rho)=sp_{\X}(x)$
for $\rho\in [0,1[$, and $(sp_{\X}\circ\tau_{\X})(x)$ is the
generic point of the stratum of the strictly semi-stable
$\widetilde{L}$-variety $\X_s$ containing $sp_{\X}(x)$. In
particular, if $E$ is a strata subset of $\X_s$, then $\Phi_{\X}$
restricts to a strong deformation retract
$$sp_{\X}^{-1}(E)\times[0,1]\rightarrow sp_{\X}^{-1}(E)$$ onto
$S(\X)\cap sp_{\X}^{-1}(E)$.

 Berkovich constructed
a natural homeomorphism $S(\X)\cong |\Delta(\X_s)|$ (here we use
the canonical homeomorphism $|\Delta(\X_s)|\cong |C(\X_s)|$
established in Proposition \ref{compar}). If $E$ is a strata
subset of $\X_s$, then this homeomorphism identifies $S(\X)\cap
sp_{\X}^{-1}(E)$ with $|\Delta_E(\X_s)|$ (cf.
\cite[5.4]{berk-contract}).

We will give an explicit description of the composed map
$$\tau_{\X(r)}:\X(r)_\eta\rightarrow S(\X(r))\cong |\Delta(\X_s)|$$
if $\X$ is a strictly semi-stable formal $K^o$-scheme, and $r\in
[0,1[$.

 For $n\geq 0$ and $q\in [0,1[$, put
$$\Sigma^n_q=\{(u_0,\ldots,u_n)\in [0,1]^{[n]}\,|\,\prod_{i\in
[n]}u_i=q\}$$ In \cite[\S\,4]{berk-contract}, Berkovich constructs
a natural homeomorphism $\alpha:\Delta_n\rightarrow \Sigma^n_q$.
It identifies the face of $\Delta_n$ corresponding to a non-empty
subset $S$ of $[n]$, with the subspace of $\Sigma^n_q$ defined by
$u_i=1$ for $i\notin S$ (note that this subspace is homeomorphic
to $\Sigma^{|S|-1}_q$). If $w$ is a point of $\Delta_n$, we call
the tuple $\alpha(w)$ in $\Sigma^n_q$ the $q$-coloured coordinates
of $w$.

For any non-empty finite set $I$ and any $q\in [0,1[$, $\alpha$
induces a map
$$\alpha:\{u\in [0,1]^I\,|\,\sum_{i\in I}u(i)=1\}\rightarrow \{u\in [0,1]^I\,|\,\prod_{i\in
I}u(i)=q\}$$ by choosing a bijection $I\cong [n]$ for some $n\geq
0$; the map is independent of this choice.

\begin{definition}[$q$-coloured representation]
We fix a value $q\in [0,1[$.  Let $X$ be a strictly semi-stable
$\widetilde{K}$-variety, and let $z$ be a point of $|\Delta(X)|$
with barycentric representation $(x,v)$ (see Definition
\ref{bary}). We define the $q$-coloured coordinates of $z$ as the
image of $v$ under the map
$$\alpha:\{u\in [0,1]^{\Psi(x)}\,|\,\sum_{i\in I}u(i)=1\}\rightarrow \{u\in [0,1]^{\Psi(x)}\,|\,\prod_{i\in
I}u(i)=q\}$$ and we call $(x,\alpha(v))$ the $q$-coloured
representation of $z$.
\end{definition}
\begin{lemma}\label{tau}
Let $\X$ be a strictly semi-stable formal $K^o$-scheme, and fix a
value $r\in [0,1[$. Let $z_0$ be any point of the special fiber
$\X_s$, and let $x$ be the unique point of $Str(\X_s)$ such that
$z_0$ is contained in the stratum $S_x$.

For each element $C$ of $\Psi(x)$, we choose a generator $T_C$ of
the kernel of the natural morphism
$\mathcal{O}_{\X,z_0}\rightarrow \mathcal{O}_{C,z_0}$. Then the
image of a point $z\in sp^{-1}_{\X(r)}(z_0)$ under the retraction
$\tau_{\X(r)}:\X(r)_\eta\rightarrow S(\X(r))\cong |\Delta(\X_s)|$
is the point with $|\pi|_{K_r}$-coloured representation
$(x,(|T_i(z)|)_{i\in \Psi(x)})$.
\end{lemma}
Recall that $\pi$ is a generator of the maximal ideal of $K^o$.
\begin{proof}
If $\X$ is of the form $\mathrm{Spf}\,A$ with
$A=K^o\{x_0,\ldots,x_m\}/(x_0.\ldots.x_p-\pi)$, then this follows
immediately from the constructions in \cite[\S\,5]{berk-contract}.
So let us assume that $\X$ admits an \'etale morphism
$h:\X\rightarrow \mY=\mathrm{Spf}\,A$. Shrinking $\X$ around
$z_0$, we may assume that $x$ is the unique maximal element of
$Str(\X_s)$ (w.r.t. the partial order defined in Section
\ref{simpint}), and that $h$ induces a bijection $Irr(\X_s)\cong
Irr(\mY_s)$. In this case, $h$ induces isomorphisms
$\beta:\Delta(\X_s)\cong\Delta(\mY_s)=\Delta[p]$ and
$\gamma:S(\X(r))\cong S(\mY(r))$, by Step 6 in the proof of
Theorems 5.2-4 in \cite{berk-contract}.

Since $|\varphi(z)|=1$ for any unit $\varphi$ on $\X$, the value
$|T_i(z)|$ only depends on $i$ and $z$ and not on the choice of
the generator $T_i$. Hence, we might as well take $T_i=h^*x_i$ for
$i=0,\ldots,p$ (we used the bijection $Irr(\X_s)\cong Irr(\mY_s)$
to identify $\Psi(x)$ with $\{0,\ldots,p\}$). Therefore, it only
remains to observe that the diagram
$$\begin{CD}
|\Delta(\X_s)|@>>> S(\X(r))
\\ @V\beta VV @VV\gamma V
\\ |\Delta(\mY_s)| @>>> S(\mY(r))
\end{CD}$$
commutes, where the horizontal arrows are the natural
homeomorphisms constructed by Berkovich (in fact, this is the
definition of the upper horizontal arrow in Step 6 of Berkovich'
proof \cite[p.\,48]{berk-contract}).
\end{proof}

\subsection{Restriction to irreducible components}
Let $X$ be a strictly semi-stable $\widetilde{K}$-variety, and let
$E$ be a union of irreducible components of $X$. Of course, $E$
itself is again a strictly semi-stable $\widetilde{K}$-variety,
and hence, it defines an associated simplicial set $\Delta(E)$
with geometric realization $|\Delta(E)|$. In general,
$|\Delta(E)|$ is not homeomorphic to $|\Delta_E(X)|$. For
instance, when $X=\mathrm{Spec}\,\widetilde{K}[x,y]/(xy)$ and $E$
is the component $x=0$, then $|\Delta(E)|$ is a point while
$|\Delta_E(X)|$ is homeomorphic to the semi-open interval $[0,1[$.
However, it is clear from the definitions that the inclusion
$Str(E)\subset Str(X)$ induces an injective morphism of simplicial
sets $\Delta(E)\rightarrow \Delta(X)$ and a natural continuous
injection $|\Delta(E)|\rightarrow |\Delta_E(X)|$ which is a
homeomorphism onto its image.

%The closed immersion $i:E\rightarrow X$ induces a morphism of
%$K_0$-analytic spaces $i_\eta:E_\eta\rightarrow X_\eta$, and it
%follows from the description of the skeleton in Section \ref{berk}
%that $i_\eta(S(E))\subset S(X)$.
Denote by $\mE$ the formal
completion of $X$ along $E$; this is a special formal
$\widetilde{K}$-scheme. The closed immersion of special formal
$\widetilde{K}$-schemes $h:E\rightarrow \mE$ induces a morphism of
$K_0$-analytic spaces $h_\eta:E_\eta\rightarrow \mE_\eta$.
\begin{prop}\label{deform}
If $X$ is a strictly semi-stable $\widetilde{K}$-variety, and $E$
is a union of irreducible components of $X$, then there exists a
strong deformation retract
$$\Phi^X_E:|\Delta_E(X)|\times[0,1]\rightarrow |\Delta_E(X)|$$ of
$|\Delta_E(X)|$ onto $|\Delta(E)|$ such that the diagram
$$\begin{CD}
E_\eta@>h_\eta >> \mE_\eta
\\@VV\tau_EV @VV\tau_{X}V
\\ |\Delta(E)|@<\Phi^X_E(\cdot,1)<< |\Delta_E(X)|
\end{CD}$$
commutes.
\end{prop}
\begin{proof}
%The vertices of $\Delta(X)$ correspond canonically to the
% of $X$. We attach to each
% $X_i$ a coordinate function $u(i)$, taking values in $[0,1]$. For each element $u(i)$ of $[0,1]^I$, we define
% its support $Supp(u)$ as the set of indices $i\in I$ with $u(i)\neq 0$. Recall that, for any non-empty subset
% $J$ of $I$, we put $X_J=\cap_{i\in J}X_i$.  Then
% $|\Delta(X)|$ is canonically homeomorphic to
% $$\{u\in [0,1]^{I}\,|\,\sum_{i\in I}u(i)=1\mbox{ and
% }X_{Supp(u)}\neq \emptyset\}$$
% (here we use the assumption that $X$ is strongly semi-stable).
% If $x\in Str(X)$ is the generic point of $X_J$, then the
% corresponding cell $C_x$ is given by the subset of tuples $u$ with
% support $J$. Moreover, if , then
% $|\Delta(E)|$ is identified with the
% subset of tuples $u$ with $Supp(u)\subset I_E$, and $|\Delta_E(X)|$ with the subset of tuples $u$ with
% $Supp(u)\cap I_E\neq \emptyset$.
We may assume $E\neq X$. Choose a sequence
$$Irr(E)=I(0)\subset I(1)\subset \ldots \subset I(q+1)=Irr(X)$$ such that
$|I(i+1)|=|I(i)|+1$ for $i=0,\ldots,q$, and put
$E^{(i)}=\cup_{V\in I(i)}V$. Denote by $\mE^{(i)}$ the formal
completion of $E^{(i)}$ along $E$, and by
$g^{(i)}:\mE^{(i)}\rightarrow \mE^{(i+1)}$ the natural closed
immersion. It suffices to construct a strong deformation retract
$\Phi^{X}_{E^{(q)}}$ of $|\Delta_{E}(X)|$ onto
$|\Delta_E(E^{(q)})|$ such that the diagram
\begin{equation}\label{diag}\begin{CD} \mE^{(q)}_\eta@>g^{(q)}_\eta >> \mE_\eta
\\@VV\tau_{E^{(q)}}V @VV\tau_{X}V
\\ |\Delta_E(E^{(q)})|@<\Phi^X_{E^{(q)}}(.,1)<< |\Delta_E(X)|
\end{CD}\end{equation}
commutes: iterating the construction, we get strong deformation
retracts $\Phi^{E^{(i+1)}}_{E^{(i)}}$ for $i=0,\ldots,q$, and
these can be glued to obtain $\Phi^X_E$.

To simplify notation, we will denote $E^{(q)}$ by $Y$, $\mE^{(q)}$
by $\mY$, and $g^{(q)}$ by $g$. Denote the unique element in
$Irr(X)\setminus Irr(Y)$ by $a$.

For any point $z$ of $|\Delta(X)|$, we will denote its
$0$-coloured representation by $(x_z,u_z)$.
%where $x_z$ is the unique point in $Str(X)$ such that
%$z$ is contained in the corresponding cell $C_{x_z}$ of
%$|\Delta(X)|$, and where $u_z\in \, ]0,1]^{\Psi(x_z)}$ are the
%barycentric coordinates of $z$ .
The point $z$ belongs to $|\Delta_E(X)|$ (resp. $|\Delta_E(Y)|$)
iff $\Psi_X(x_z)\cap Irr(E)\neq \emptyset$ (resp. iff
$\Psi_X(x_z)\cap Irr(E)\neq \emptyset$ and $\Psi_X(x_z)\subset
Irr(Y)$).

For $\emptyset\neq J\subset Irr(X)$, $x$ a generic point of $X_J$,
and $u\in [0,1]^J$, we define the support $Supp(u)$ of $u$ as the
set of indices $i\in J$ with $u(i)\neq 1$. We identify $(x,u)$
with the point $z\in |\Delta(X)|$ with $0$-coloured representation
$(x_z,u_z)$, were $u_z$ is the restriction of $u$ to $Supp(u)$ and
$x_z$ is the unique generic point of $X_{Supp(u)}$ such that $x$
belongs to the Zariski closure of $\{x_z\}$ in $X$.

 Consider the function %(expressed in $0$-coloured representations)
$$\Phi^X_{Y}:|\Delta_E(X)|\times [0,1]\rightarrow
|\Delta_E(X)|:(z=(x_z,u_z),\rho)\mapsto
(x_z,u(\rho))=\Phi^X_{Y}(z,\rho)$$ where $u(\rho)\in
[0,1]^{\Psi_X(x_z)}$ depends on $x_z,\,u_z$ and $\rho$, and is
defined in the following way.

\textit{Case 1.} If $a\notin \Psi_X(x_z)$, then $u(\rho)=u_z$ for
each value of $\rho$.

\textit{Case 2.} Now assume that $a\in \Psi_X(x_z)$. If there
exists an element $b\in \Psi_Y(x_z)=\Psi_X(x_z)\cap Irr(Y)$ with
$u_z(b)=0$, then we put
$$ u(\rho)(i)=\left\{\begin{array}{l}u_z(i) \mbox{ for } i\in \Psi_Y(x_z) \mbox{ and } \rho\in
[0,1] \\u_z(a)\mbox{ for }i=a\mbox{ and }\rho \in [0,1/2]
\\u_z(a)+(1-u_z(a))(2\rho-1)\mbox{ for }i=a\mbox{ and }\rho \in [1/2,1]
\end{array}\right.$$
\textit{Case 3.} Finally, suppose that $u_z(i)\neq 0$ for all
$i\in \Psi_Y(x_z)$. Then necessarily $u_z(a)=0$. We put
$u(\rho)(a)=0$ for $\rho\in [0,1/2]$, and
%$$u(\rho)(i)=u_z(i)+2\rho.\min_{i\in \Psi(x)}\left(\frac{u_z(i)}{1-u_z(i)}\right).(u_z(i)-1) $$
$$u(\rho)(i)=u_z(i)-2\rho\cdot\min_{j\in \Psi_X(x_z)}\left(u_z(j)\right) $$
for $i\in \Psi_Y(x_z)$ and $\rho\in [0,1/2]$. Then there exists an
element $b\in \Psi_Y(x_z)$ with $u(1/2)(b)=0$, so that the
definition in Case 2 applies to $v:=u(1/2)\in
[0,1[^{\Psi_X(x_z)}$, and we put $u(\rho)=v(\rho)$ for $\rho\in
[1/2,1]$.

%$u(i)=(1-\rho)u_z(i)$ if $i\notin I_E$, and
%$$u(i)=u_z(i).\frac{1-(1-\rho)\sum_{j\notin I_E}u_z(j)}{\sum_{\ell\in \Psi(x)\cap I_E}u_z(\ell)}$$
%if $i\in I_E$. Note that the denominator $\sum_{\ell\in
%\Psi(x)\cap I_E}u_z(\ell)$ is non-zero since $\Psi(x)\cap I_E\neq
%\emptyset$.
It is easily seen that the map $\Phi^X_{Y}$ is continuous, and
defines a strong deformation retract of $|\Delta_E(X)|$ onto
$|\Delta_E(Y)|$. Let us check the commutativity of diagram
(\ref{diag}). Let $z$ be any point of $\mY_\eta$, and put
$sp_{\mY}(z)=sp_X(z)=z_0$.
%We denote by
%$\Psi_{Y}(z_0)$ and $\Psi_X(z_0)$ the sets of irreducible
%components of $Y$, resp. $X$, passing through $z_0$; we have
%$\Psi_{Y}(z_0)= \Psi_X(z_0)\cap Irr(Y)$.
For each element $C$ of $\Psi_X(z_0)$, we choose a generator $T_C$
of the kernel of the natural morphism
$\mathcal{O}_{X,z_0}\rightarrow \mathcal{O}_{C,z_0}$. Denote by
$y$ and $x$ the points of $Str(Y)$, resp. $Str(X)$ such that $z_0$
belongs to the corresponding stratum of $Y$, resp. $X$. Then by
Lemma \ref{tau} we have to show that $\Phi^X_{Y}(\cdot,1)$ takes
the point $z'=(\tau_X\circ g_\eta)(z)$ of $|\Delta_E(X)|$ with
$0$-coloured representation $(x,(|T_i(z)|)_{i\in \Psi_X(z_0)})$ to
the point $\tau_{Y}(z)$ of $|\Delta_E(Y)|$ with $0$-coloured
representation $(y,(|T_i(z)|)_{i\in \Psi_{Y}(z_0)})$.

If $\Psi_X(z_0)=\Psi_Y(z_0)$ this is obvious, so assume that
$\Psi_X(z_0)=\Psi_Y(z_0)\sqcup \{a\}$. Since $z$ is an element of
$\mY_\eta$, there exists an element $i$ of $\Psi_{Y}(z_0)$ with
$|T_i(z)|=0$. By Case 2 of the definition, $\Phi^X_Y(z',1)$ is
given by the couple $(x,v)$ with $v\in [0,1]^{\Psi_X(z_0)}$,
$v(i)=|T_i(z)|$ for $i\neq a$, and $v(a)=1$. By the
identifications we made, this is exactly the point $\tau_Y(z)$
(since $y$ is the unique generic point of $X_{\Psi_{Y}(z_0)}$ such
that $x$ belongs to the Zariski closure of $\{y\}$ in $Y$).
%
%Denote by $(x,v)$ the $0$-coloured representation of
%$\Phi_{Y}^X(z',1)$. By definition of the map $\Phi^X_{Y}$, $x$ is
%the unique generic point of $\cap_{Z\in \Psi(x_X)\cap Irr(Y)}Z$
%with $x_X\in \overline{\{x\}}$, and this is exactly the point
%$x_{Y}$. This completes the proof.
\end{proof}
%\begin{remark}
%Of course, it is possible to construct a much more elementary
%strong deformation retract of $|\Delta(X)|$ onto $|\Delta_E(X)|$
%using barycentric representations. The reason for our definition
%is the commutativity property in the statement of the proposition.
%\end{remark}
\begin{cor}\label{ssvar}
The map $h_\eta:E_\eta\rightarrow \mE_\eta$ is a homotopy
equivalence.
\end{cor}
\begin{proof}
By Proposition \ref{deform}, $\Phi^X_E(\cdot,1)$ is a homotopy
equivalence with as homotopy inverse the natural embedding
$i:|\Delta(E)|\rightarrow |\Delta_E(X)|$. Moreover, $\tau_E$ and
$\tau_X$ are homotopy equivalences with as homotopy inverses the
natural embeddings $\sigma_E:|\Delta(E)|\rightarrow E_\eta$, resp.
$\sigma_X:|\Delta_E(X)|\rightarrow \mE_\eta$ (we used Berkovich'
natural homeomorphisms to identify $|\Delta(E)|$ with the skeleton
$S(E)$ and $|\Delta_E(X)|$ with $S(X)\cap sp_{X}^{-1}(E)$). Now
the fact that $h_\eta$ is a homotopy equivalence follows from the
commutativity of the diagram in the statement of Proposition
\ref{deform}, and the commutativity of the diagram
$$\begin{CD}
|\Delta(E)|@>i>> |\Delta_E(X)|
\\ @V\sigma_E VV @VV\sigma_X V
\\ E_\eta @>h_\eta>> \mE_\eta
\end{CD}$$ which follows easily from the description of the
skeleton in Section \ref{berk}.
\end{proof}
%%Then $\Delta(X)$ is defined as the nerve of $\Xi(x)$. This means
%%that, for any integer $n\geq 0$, $\Delta(X)([n])$ is the set of
%%non-decreasing maps $[n]\rightarrow \Xi(x)$, and that
%%$\Delta(X)([m]\rightarrow [n])$ is the map given by composition.

\section{Homotopy type of the analytic Milnor
fiber}\label{sec-hom} Let $k$ be any field, and put $R=k[[t]]$ and
$K=k((t))$, endowed with the $t$-adic valuation (so $R=K^o$). For
each $0<r<1$, we denote by $|\cdot|_r$ the $t$-adic absolute value
on $K$ with $|t|_r=r$. We fix an algebraic closure $K^a$ of $K$.

% For any special formal $R$-scheme
%$\X$, we denote its special fiber by $\X_s$ and its generic fiber
%by $\X_\eta$; the former is a separated $k$-scheme of finite type,
%while the latter is a strictly $K$-analytic space. There exists a
%natural specialization morphism of locally ringed spaces
%$sp:\X_\eta\rightarrow \X$.

We endow $k$ with its trivial absolute value $|\cdot|_0$, and we
put $\mR:=\mathrm{Spf}\,R$. Moreover, we endow $k[t]$ with the
trivial Banach norm (this norm coincides with the Gauss norm if we
view $k[t]$ as the algebra of convergent power series $k\{t\}$).
The formal scheme $\mR$ is a special formal $k$-scheme in the
sense of \cite{berk-vanish2}. We denote its generic fiber by
$\mR_\eta$; it coincides with the open unit disc $D_0(1)=\{x\in
\mathscr{M}(k[t])|\ |t(x)|<1\}$.

The following result was stated in Step 3 of the proof of
\cite[Thm\,4.1]{berk-limit}, without proof. We include the
elementary proof for the reader's convenience.
\begin{lemma}\label{dischomeo}
The natural map
$$\Psi:D_0(1)\rightarrow [0,1[~:x\mapsto |t(x)|$$ is a
homeomorphism. If we put $p_r:=\Psi^{-1}(r)$ for $0\leq r<1$, then
the residue field $\mathscr{H}(p_r)$ is $K_0=(k,|\cdot|_0)$ for
$r=0$, and $K_r=(K,|\cdot|_r)$ for $0<r<1$.
\end{lemma}
\begin{proof}
The map $\Psi$ is obviously continuous. Its inverse is
$$\Psi^{-1}:[0,1[\,\rightarrow D_0(1):r\mapsto p_r$$
where $p_r$ is bounded multiplicative semi-norm in
$\mathscr{M}(k[t])$ defined by $p_r(\sum_{i=0}^{n}a_i
t^i):=\max_{i}|a_i|_0 r^i$ (with the convention that $0^0=1$).

Indeed: it is clear that $\Psi\circ \Psi^{-1}$ is the identity.
Now we show that $\Psi^{-1}$ is also left inverse to $\Psi$.
Choose $x$ in $D_0(1)$ and put $r=|t(x)|$.
%First, note that
%$|a(x)|=|a|_0$ for $a\in k$, since for $a\neq 0$ both $x(a)$ and
%$x(a^{-1})$ are bounded by $1$. Now
A classical trick shows that the residue field $\mathscr{H}(x)$ of
$x$ is ultrametric: for $f,g$ in $k[t]$, and $n\geq 0$, we have
\begin{eqnarray*}|(f+g)^n(x)|=|(\sum_{i=0}^n \binom{n}{i}f^ig^{n-i})(x)|\leq
\sum_{i=0}^n |\binom{n}{i}|_0\, |f(x)|^i|g(x)|^{n-i}
%\leq \sum_{i=0}^n
%|f(x)|^i|g(x)|^{n-i}
\\ \leq (n+1)(\max\{\,|f(x)|,|g(x)|\,\})^n\end{eqnarray*} and taking
$n$-th roots and sending $n$ to $\infty$, we see that
$|(f+g)(x)|\leq \max\{\,|f(x)|,|g(x)|\,\}$.

If $r=0$, then the ultrametric property implies $x=p_0$, so we may
assume $0<r<1$. In this case, $|at^i(x)|\neq |a't^j(x)|$ for
$a,a'\in k^{\ast}$ and $i\neq j$, so we get
$|(\sum_{i=0}^{n}a_it^i)(x)|=\max_{i}|a_i|_0 r^i$.

It remains to prove that $\Psi^{-1}$ is continuous. By definition
of the spectral topology, it suffices to show that
$[0,1[\,\rightarrow \R_+:r\mapsto |f(p_r)|=p_r(f)$ is continuous
for each $f\in k[t]$. This, however, is clear.

Finally, we determine the residue fields of the points $p_r$.
Since $|\sum_{i=0}^na_i t^i(x)|=|a_0|_0$, we see that
$\mathscr{H}(p_0)=(k,|\cdot|_0)$. For $r>1$, no element of
$k[t]\setminus\{0\}$ vanishes in $p_r$, so $\mathscr{H}(p_r)$ is
the completion of $k(t)$ w.r.t. $p_r$; this is exactly
$(K,|\cdot|_r)$.
\end{proof}

If $\X$ is a special formal $R$-scheme, then we can also consider
$\X$ as a special formal $k$-scheme via the composition
$\X\rightarrow \mR\rightarrow \mathrm{Spec}\,k$. We denote this
object by $\X^k$. This yields a forgetful functor
$$(SpF/R)\rightarrow (SpF/k):\X\mapsto \X^k$$ from the category
$(SpF/R)$ of special formal $R$-schemes to the category $(SpF/k)$
of special formal $k$-schemes.
 We can associate to $\X^k$ its generic fiber
$\X^k_\eta$ (a $K_0$-analytic space), and there is a natural
specialization map $sp_{\X^k}:\X^k_\eta\rightarrow \X_0$.

The map of special formal $k$-schemes $h:\X^k\rightarrow \mR$
induces a map of $K_0$-analytic spaces
$h_\eta:\X^k_\eta\rightarrow \mR_\eta$. Its fibers can be
described as follows.

\begin{lemma}
For $0\leq r<1$, there is a canonical isomorphism $$\X(r)\cong
\X^k\widehat{\times}_{\mR}\mathrm{Spf}\,\mathcal{H}(p_r)^o$$ and
the fiber of $h_\eta$ over $p_r$ is canonically isomorphic to the
$K_r$-analytic space $\X(r)_\eta$.
\end{lemma}
\begin{proof}
By the construction of the generic fiber in \cite[\S
1]{berk-vanish2}, it suffices to consider the case where
$\X=\mathrm{Spf}\,A$, with $A$ of the form
$R[[x_1,\ldots,x_m]]\{y_1,\ldots,y_n\}$. Then $\X^k_\eta$ is the
polydisc $$\{z\in \mathscr{M}(k[t,x_1,\ldots,y_n])\,|\ |t(z)|<1
\mbox{ and }|x_i(z)|<1\mbox{ for }i=1,\ldots,m\}$$ where
$k[t,x_1,\ldots,y_n]$ carries the trivial Banach norm, and the map
$h_\eta$ sends the bounded multiplicative semi-norm $z$ to its
restriction to $k[t]$. Now we observe that
$$k[t,x_1,\ldots,y_n]\widehat{\otimes}_{k[t]}\mathscr{H}(p_0)\cong
k[x_1,\ldots,y_n]$$ while
$$k[t,x_1,\ldots,y_n]\widehat{\otimes}_{k[t]}\mathscr{H}(p_r)\cong
K_r\{x_1,\ldots,y_n\}$$ for $0<r<1$.
\end{proof}

\begin{prop}\label{fiberhomeo}
Let $\X$ be a special formal $R$-scheme. If we denote by $\lambda$
the map $|\X^k_\eta|\rightarrow |\mR_\eta|\cong [0,1[$, then for
each $r\in [0,1[$, there exists a canonical homeomorphism
$|\X(r)_\eta|\cong \lambda^{-1}(r)$ such that the square
$$\begin{CD}
|\X(r)_\eta|@>\sim >> \lambda^{-1}(r)\subset |\X^k_\eta|
\\ @V sp_{\X(r)} VV @VV sp_{\X^k} V
\\ |\X_0|@= |\X_0|
\end{CD}$$
commutes.

Moreover, for any $0<r<1$, there exists a canonical homeomorphism
$$\phi:\lambda^{-1}(\,]0,1[\,)\rightarrow |\X(r)_\eta|\times\,]0,1[$$ such that the composition
of $\phi$ with the projection
$|\X(r)_\eta|\times\,]0,1[\,\rightarrow ]0,1[$ coincides with
$\lambda: \lambda^{-1}(\,]0,1[\,)\rightarrow ]0,1[$.
%Moreover, there exists a canonical deformation retract
%$\pi:|\X^k_\eta|\times [0,1]\rightarrow |\X^k_\eta|$ of
%$|\X^k_\eta|$ onto $|\X(0)_\eta|$.
\end{prop}
\begin{proof}
It suffices to consider the case where $\X=\mathrm{Spf}\,A$ is
affine, with $A$ of the form
$$A=R[[x_1,\ldots,x_m]]\{y_1,\ldots,y_n\}/(f_1,\ldots,f_\ell) $$
Then $\X^k_\eta$ is a closed subset of the polydisc
$$E:=\{z\in \mathscr{M}(k[t,x_1,\ldots,y_n])\,|\ |t(z)|<1 \mbox{ and
}|x_i(z)|<1\mbox{ for }i=1,\ldots,m\}$$ where
$k[t,x_1,\ldots,y_n]$ carries the trivial Banach norm, and this
closed subset is defined by the equations $z(f_j)=0$,
$j=1,\ldots,\ell$. The map $\lambda$ sends $z$ to the point
$|t(z)|$ of $|\mR_\eta|\cong [0,1[$.

It is clear that $\lambda^{-1}(r)$ is canonically homeomorphic to
$|\X(r)_\eta|$ for each $r$, and that this homeomorphism is
compatible with the specialization maps $sp_{\X(r)}$ and
$sp_{\X^k}$: w.r.t. both of these maps, the image of a point $z$
in $\lambda^{-1}(r)$ is the open prime ideal $\{f\in A\,|\
|f(z)|<1\}$ of $A$.

 Now fix
$r$ in $]0,1[$, and consider the map
$$\Psi:\lambda^{-1}(\,]0,1[\,)\rightarrow |\X(r)_\eta|\times\,]0,1[:x\mapsto (x^{\log_{\lambda(x)} r},\lambda(x))$$
where we denote by $x^{\log_{\lambda(x)} r}$ the bounded
multiplicative semi-norm in $\mathscr{M}(k[t,x_1,\ldots,y_n])$
sending
 $f\in
k[t,x_1,\ldots,y_n]$ to $x(f)^{\log_{\lambda(x)} r}$; then clearly
$\lambda(x^{\log_{\lambda(x)} r})=r$. The map $\Psi$ is a
bijection, with inverse
$$\Psi^{-1}: |\X(r)_\eta|\times\,]0,1[ \rightarrow
\lambda^{-1}(\,]0,1[\,): (x,\rho)\mapsto x^{\log_{r} \rho}$$ and
one checks immediately that both $\Psi$ and $\Psi^{-1}$ are
continuous.
%Now we define the map $\pi$ by
%$$\pi:|\X^k_\eta|\times [0,1]\rightarrow
%|\X^k_\eta|:(z,\rho)\mapsto \pi(z,\rho)$$ with
%$|f(\pi(z,\rho))|=|f(z)|^{1/(1-\rho)}$ for each $f$ in
%$k[t,x_1,\ldots,y_n]$. By convention, $a^{\infty}=0$ for $0\leq
%a<1$ and $1^{\infty}=1$. It is clear that $\pi$ is continuous,
%that $\pi(z,0)=z$ for each point $z$ of $\X^k_\eta$, that
%$\pi(y,\rho)=y$ for each couple $(y,\rho)$ in $\X(0)_\eta\times
%[0,1]$, and that
\end{proof}
\begin{prop}\label{homeq}
Let $\X$ be a strictly semi-stable formal $R$-scheme, let $E$ be a
union of irreducible components of $\X_s$, and denote by $\mE$ the
formal completion of $\X$ along $E$.

(1) The inclusion map
$$\mE(r)_\eta\rightarrow \mE_\eta^k$$
is a homotopy equivalence, for each $r\in [0,1[$.

(2) The natural map $E_\eta\rightarrow \mE(0)_\eta$, induced by
the closed immersion $E\rightarrow \mE(0)$, is a homotopy
equivalence. In particular, if $E$ is proper, then
$E^{an}\rightarrow \mE(0)_\eta$ is a homotopy equivalence.
\end{prop}
\begin{proof}
(1) In the last part of the proof of \cite[4.1]{berk-limit},
Berkovich constructs the so-called skeleton
$S(\X/\mathfrak{R})\subset \X^k_\eta$ associated to the morphism
of special formal $k$-schemes $\X^k\rightarrow \mathfrak{R}$: it
is the union of the skeletons $S(\X(r))\subset \X(r)_\eta\cong
\lambda^{-1}(r)$, $r\in [0,1[$. Moreover, he shows that the map
$\Phi:\X^k_\eta\times[0,1]\rightarrow \X^k_\eta$ which coincides
with $\Phi_{\X(r)}$ on $\lambda^{-1}(r)\times [0,1]$,
 is a strong deformation
retract of $\X^k_\eta$ onto $S(\X/\mathfrak{R})$. By
\cite[5.2(iv)]{berk-contract} $\Phi$ restricts to a strong
deformation retract of $(sp_{\X^k})^{-1}(E)$ onto
$S(\X/\mathfrak{R})\cap (sp_{\X^k})^{-1}(E)$ compatible with
$\lambda$ (i.e. it is a strong deformation retract on each fiber
of $\lambda$).

Next, Berkovich states that there exists a homeomorphism
$|\Delta(\X_s)|\times \mathfrak{R}_\eta\rightarrow
S(\X/\mathfrak{R})$ such that the projection of
$|\Delta(\X_s)|\times \mathfrak{R}_\eta$ on the second factor
$\mathfrak{R}_\eta$ corresponds to the map $\lambda$ on
$S(\X/\mathfrak{R})$. His construction contains a minor error, but
it can easily be corrected as follows. By the same arguments, it
suffices to construct a homeomorphism
$$f:\Sigma^n_0\times [0,1[\,\rightarrow \{(x,\rho)\in \R^{[n]} \times [0,1[\,|\,x\in \Sigma^n_{\rho}\}$$ for each $n$,
such that these homeomorphisms are compatible with the face maps
(so that we get good gluing properties). We can define $f$ by
sending $(y,\rho)$ to $(x,\rho)$, where $x$ is the unique
intersection point of $\Sigma^n_{\rho}$ and the segment in
$\R^{[n]}$ joining $y$ and $(1,\ldots,1)$. This homeomorphism
identifies $S(\X/\mathfrak{R})\cap (sp^k)^{-1}(E)$ with
$|\Delta_E(X)|\times \mathfrak{R}_\eta$. This proves (1).

(2) Since $\X_s$ is a strictly semi-stable $k$-variety, and
$\mE(0)$ is isomorphic to the completion of $\X_s$ along $E$, this
follows immediately from Corollary \ref{ssvar} and the fact that
$E_\eta\cong E^{an}$ for proper $E$.
\end{proof}

\begin{lemma}\label{extension}
Assume that $k$ is algebraically closed. If $\X$ is a strictly
semi-stable formal $R$-scheme, and $E$ is any strata subset of
$\X_s$, then the natural map
$$sp_{\X(r)}^{-1}(E)\widehat{\times}_{K_r} L\rightarrow sp_{\X(r)}^{-1}(E) $$
is a homotopy equivalence for any $r\in \, ]0,1[$ and any
isometric embedding of non-archimedean fields $K_r\subset L$.
\end{lemma}
\begin{proof}
Put $\mY=\X\widehat{\times}_{\mathrm{Spf}\,R}\mathrm{Spf}\,L^o$.
Then $\mY_\eta\cong \X(r)_\eta\widehat{\times}_{K_r} L$,
$\mY_s\cong \X_s\times_k \widetilde{L}$, and these natural
isomorphisms commute with the specialization maps $sp_{\X(r)}$ and
$sp_{\mY}$, so that they induce a natural isomorphism
$$sp_{\X(r)}^{-1}(E)\widehat{\times}_{K_r} L \cong
sp_{\mY}^{-1}(F)$$ where $F$ denotes the inverse image of $E$ in
$\mY_s$.

 Since $k$ is algebraically
closed, the natural map $h_s:\mY_s\rightarrow \X_s$ induces a
bijection $Irr(\mY_s)\cong Irr(\X_s)$ and a homeomorphism
$\alpha:|\Delta(\mY_s)|\cong |\Delta(\X_s)|$ identifying
$|\Delta_{F}(\mY_s)|$ with $|\Delta_E(\X_s)|$. If we denote by $h$
the natural morphism $\mY_\eta\rightarrow \X(r)_\eta$, it is easy
to see from the description in Section \ref{berk} that the diagram
$$\minCDarrowwidth20pt\begin{CD}
sp_{\mY}^{-1}(F)@>\tau_{\mY} >> |\Delta(\mY_s)|@>\cong
>>S(\mY)\cap sp_{\mY}^{-1}(F)@>>> sp_{\mY}^{-1}(F)
\\ @VhVV @V\alpha V\cong V @V\cong VhV @VVhV
\\sp^{-1}_{\X(r)}(E)@>\tau_{\X(r)} >> |\Delta(\X_s)|@>\cong >> S(\X(r))\cap
sp_{\X(r)}^{-1}(E)@>>> sp_{\X(r)}^{-1}(E)
\end{CD}$$
commutes (the right horizontal arrows are the inclusion maps).
%But $h$ and $h_0$ commute with the specialization maps
%$sp_{\X^L}$ and $sp_{\X(r)}$, and $h$ commutes with the
%retractions $\tau_{\X^L}$ and $\tau_{\X(r)}$ (this is obvious if
%$\X$ is as in Case 1 of Section \ref{berk}, and can be generalized
%to arbitrary $\X$ by using the same arguments as in the proof of
%Lemma \ref{tau}). As a result, we get a commutative diagram
\end{proof}

\begin{prop}\label{longexact}
Assume that $k$ is an algebraically closed field of characteristic
zero, and fix $r\in \, ]0,1[$. Let $X$ be a proper flat
$R$-variety such that $X\times_R K$ is smooth over $K$, and such
that $X_s$ has at most one singular point $x$, and denote by $\X$
its $t$-adic completion. Then there exists a canonical long exact
sequence in integral singular cohomology
$$\minCDarrowwidth10pt\begin{CD}
\ldots @>>> H^i((X_s)^{an},\Z)@>>>
H^i(\overline{\X(r)}_\eta,\Z)@>i^*>>
\widetilde{H}^i(\,\overline{]x[}\,,\Z)@>>>
H^{i+1}((X_s)^{an},\Z)@>>>\ldots
\end{CD} $$
with $]x[= sp_{\X(r)}^{-1}(x)$, and where
$i:\overline{]x[}\rightarrow \overline{\X(r)}_\eta$ is the
inclusion map and $\widetilde{H}^*(\cdot)$ is reduced cohomology.
\end{prop}
\begin{proof}
 Passing to a finite extension of $R$, we may assume that there
 exists a proper morphism of $R$-varieties $h:Y\rightarrow X$ such that
 $Y$ is strictly semi-stable, such that $h$ is an isomorphism over
 the complement of $x$ in $X$, and such that $E:=h^{-1}(x)$ is a union of irreducible components
 of $Y_s$. We denote by $\mY$ the $t$-adic
 completion of $Y$, and by $\mE$ the formal completion
of $Y$ along $E$.

The morphism $h$ induces a surjective morphism of $k$-analytic
spaces $h_s^{an}:(Y_s)^{an}\rightarrow (X_s)^{an}$; since $X$ and
$Y$ are proper over $R$, $(Y_s)^{an}$ and $(X_s)^{an}$ are compact
Hausdorff spaces. Moreover, $h_s^{an}$ maps $(Y_s\setminus
E)^{an}\cong (Y_s)^{an}\setminus E^{an}$ isomorphically to
$(X_s)^{an}\setminus \{x\}$, and maps $E^{an}$ to $\{x\}$.
Therefore, $h_s^{an}$ induces a homeomorphism
$(Y_s)^{an}/E^{an}\approx (X_s)^{an}$, and we get a natural exact
sequence
$$\minCDarrowwidth10pt\begin{CD}
\ldots @>>> H^i((X_s)^{an},\Z)@>>> H^i((Y_s)^{an},\Z)@>>>
\widetilde{H}^i(E^{an},\Z)@>>> H^{i+1}((X_s)^{an},\Z)@>>>\ldots
\end{CD} $$
By Proposition \ref{homeq} and Lemma \ref{extension}, the natural
maps $\overline{\X(r)}_\eta\rightarrow \mY^k_\eta$,
$(Y_s)^{an}\rightarrow \mY^k_\eta$, $E^{an}\rightarrow \mE^k_\eta$
and $\overline{]x[}\cong \overline{\mE(r)_\eta}\rightarrow
\mE^k_\eta$ are all homotopy equivalences. Hence, we obtain the
desired exact sequence.
\end{proof}

Now we come to the main result of this section: the description of
the homotopy type of the analytic Milnor fiber. First, we need an
auxiliary definition. Let $r\in\,]0,1[$ be the value fixed in
Section \ref{subsec-anmil} to define the analytic Milnor fiber
$\mathscr{F}_x$ as a $K_r$-analytic space.

\begin{definition}[Strictly semi-stable model]\label{red}
 Let $X$ be a variety over
$k$, endowed with a morphism $f:X\rightarrow \mathrm{Spec}\,k[t]$
which is flat over the origin and has smooth generic fiber, and
let $x$ be a closed point of the special fiber $X_s=f^{-1}(0)$. A
strictly semi-stable model of the germ $(f,x)$ of $f$ at $x$
consists of the following data:
\begin{enumerate}
\item an integer $d>0$, and an embedding of $k[t]$-algebras
$$A_d=k[t,u]/(u^d-t)\rightarrow K^a$$
 \item
a flat projective morphism $g:Y\rightarrow \mathrm{Spec}\,A_d$
whose $u$-adic completion is strictly semi-stable,
%whose $t$-adic completion is strictly semi-stable,
\item open subschemes $U$ and $V$ of $X$, resp. $Y$, with $x\in
U$, \item a proper morphism $\varphi:V\rightarrow
U\times_{k[t]}A_d$ which is an isomorphism over the complement of
$U_s$, such that $g=p_2\circ \varphi$ on $V$ (with $p_2$ the
projection $U\times_{k[t]}A_d\rightarrow \mathrm{Spec}\,A_d$) and
such that $\varphi^{-1}(x)$ is a union of irreducible components
of the special fiber $Y_s$ of $g$.
\end{enumerate}
We'll denote this strictly semi-stable model by $(Y,g,\varphi)$
(the other data are implicit in the notation). We call $d$ the
ramification index of the strictly semi-stable model.

 A strictly
semi-stable model of the analytic Milnor fiber $\mathscr{F}_x$ of
$f$ at $x$ consists of the following data:
\begin{enumerate}
\item an integer $d>0$, and an embedding of non-archimedean
$K_r$-fields
$$K_r(d)=K_r[u]/(u^d-t)\rightarrow K_r^a$$
 \item
a strictly semi-stable formal $K_r(d)^o$-scheme $\mY$, and a
closed subvariety $E$ of $\mY_s$ which is a union of irreducible
components of $\mY_s$,
%whose $t$-adic completion is strictly semi-stable,
\item an isomorphism of $K_r(d)$-analytic spaces
$\varphi:sp_{\mY(r)}^{-1}(E)\cong \mathscr{F}_x\times_{K_r}
K_r(d)$.
\end{enumerate}
We'll denote this strictly semi-stable model by $(\mY,E,\varphi)$
(the other data are implicit in the notation). We call $d$ the
ramification index of the strictly semi-stable model.
\end{definition}
It is clear that any strictly semi-stable model of $(f,x)$ induces
a strictly semi-stable model of $\mathscr{F}_x$ by passing to the
$t$-adic completion (and of course, this still holds if we omit
the projectivity condition in (2)).
\begin{prop}
Suppose that $k$ has characteristic zero. Let $X$ be a variety
over $k$, endowed with a morphism $f:X\rightarrow
\mathrm{Spec}\,k[t]$ which is flat over the origin and has smooth
generic fiber, and let $x$ be a closed point of the special fiber
$X_s=f^{-1}(0)$. Then $(f,x)$ admits a strictly semi-stable model.
\end{prop}
\begin{proof}
We may as well assume that $f$ is projective: restrict $f$ to an
affine neighbourhood of $x$, consider its projective completion,
and resolve singularities at infinity. Now the result follows from
the semi-stable reduction theorem \cite[II]{Kempf}.
% (see also
%\cite[Thm\,12.3]{peters-steenbrink}).
\end{proof}
\begin{cor}
Under the same conditions, $\mathscr{F}_x$ admits a strictly
semi-stable model.
\end{cor}
\begin{theorem}\label{hom-mil}
Suppose that $k$ is algebraically closed (of arbitrary
characteristic). Let $X$ be a variety over $k$, endowed with a
morphism $f:X\rightarrow \mathrm{Spec}\,k[t]$ which is flat over
the origin and has smooth generic fiber, and let $x$ be a closed
point of the special fiber $X_s=f^{-1}(0)$. Suppose that
$\mathscr{F}_x$ admits a strictly semi-stable model
$(\mY,E,\varphi)$. Then $\overline{\mathscr{F}}_{\!\! x}$ is
naturally homotopy-equivalent to $|\Delta(E)|$. In particular, the
homotopy type of $|\Delta(E)|$ does not depend on the chosen
strictly semi-stable model.
\end{theorem}
\begin{proof}
Let $d$ be the ramification index of the strictly semi-stable
model $(\mY,E,\varphi)$.
 The isomorphism $\varphi$ induces an isomorphism
$$sp^{-1}_{\mY(r)}(E)\widehat{\times}_{K_r(d)}\widehat{(K_r)^a}\cong
\overline{\mathscr{F}}_{\!\! x}$$ so the result follows from Lemma
\ref{extension} and Proposition \ref{homeq}.
\end{proof}
\begin{remark}
 By the same arguments, we have the following result: suppose
that $k$ is algebraically closed, and fix $r\in \,]0,1[$. Consider
a generically smooth $stft$ formal $R$-scheme $\X$, and assume
that it admits a strictly semi-stable model $h:\mY\rightarrow \X$
(i.e. $\mY$ is a strictly semi-stable formal $L^o$-scheme for some
finite extension $L$ of $K_r$ in $K_r^a$, and $h$ is a morphism of
formal $R$-schemes such that the induced morphism
$\mY_\eta\rightarrow \X(r)_\eta\times_{K_r}L$  is an isomorphism).
Such a model exists, in particular, if $k$ has characteristic zero
(use embedded resolution for singularities for $(\X,X_s)$ as in
\cite{temkin-resol} and apply the algorithm for semi-stable
reduction in characteristic zero \cite[II]{Kempf}).

The analytic space $\overline{\X(r)}_\eta$ is naturally
homotopy-equivalent to $|\Delta(\mY_s)|$. In particular, the
homotopy type of $|\Delta(\mY_s)|$ does not depend on the chosen
strictly semi-stable model. This result, and the one in Theorem
\ref{hom-mil}, are similar in nature to
\cite[Thm\,4.8]{thuillier}.
\end{remark}

\if false
\section{Arc spaces and trivial valuations}\label{sec-arc}
Let $k$ be any field. We put $R=k[[t]]$ (with the $t$-adic
topology) and $K=k((t))$ (with its $t$-adic valuation) as in
Section \ref{sec-hom}.

We endow $k$ with its trivial absolute value $|\cdot|_0$. By
non-Archimedean GAGA \cite[3.5]{Berk1}, we can associate to any
separated $k$-scheme of finite type $X$ a $k$-analytic space
$X^{an}$. There exists a natural morphism of locally ringed spaces
$i_X:X^{an}\rightarrow X$.

If $X=\mathrm{Spec}\,A$ is affine, the topological space
$|X^{an}|$ can be described as follows: its elements are
multiplicative semi-norms $x$ on $A$ that restrict to the the
trivial absolute value on $k$. This set $|X^{an}|$ is endowed with
the coarsest topology such that the map $|X^{an}|\rightarrow
\R_+:x\mapsto x(f)$ is continuous for each $f\in A$. The value
$x(f)$ is often denoted by $|f(x)|$ (it is the absolute value of
the image of $f$ in the residue field $\mathscr{H}(x)$ of $x$) but
we will not adopt this notation in this section. The map
$i_X:X^{an}\rightarrow X$ sends $x$ to the prime ideal $x^{-1}(0)$
in $A$.

Since the absolute value on $k$ is trivial, $i_X$ has a natural
section $s_X$ (on the underlying sets).
 For $X=\mathrm{Spec}\,A$, the section $s_X$ maps a
point $y$ in $\mathrm{Spec}\,A$ to the composition of
$A\rightarrow A/y$ with the trivial absolute value on $A/y$. Note
that $s_X$ is not continuous, in general. If $y$ is closed, then
$s_X(y)$ is the unique point of $X^{an}$ mapped to $y$ by $i_X$.
Indeed, any such semi-norm factors through the finite extension
$A/y$ of $k$, and the prolongation of $|\cdot|_0$ to $A/y$ is
unique.

We can also view $X$ as a $stft$ formal scheme over $k^o=k$, and
we denote by $X_\eta$ its generic fiber  in the sense of
\cite{berk-vanish2}. This is an analytic domain in $X^{an}$ in a
natural way; see \cite[1.10]{thuillier} where $X_\eta$ was denoted
by $X^\beth$. Moreover, $X_\eta=X^{an}$ iff $X$ is proper over
$k$. The analytic space $X_\eta$ is endowed with a natural
specialization map $sp_X:X_\eta\rightarrow X$. It is easily seen
that the image of the section $s_X:X\rightarrow X^{an}$ is
contained in $X_\eta$, and that the section $s_X$ coincides with
the map $\sigma$ from \cite[1.9]{thuillier}. If $X$ is affine, say
$X=\Spec A$, then $X_\eta$ is the $k$-affinoid space
$\mathscr{M}(A)$, where we endow $A$ with the trivial Banach norm.

The arc scheme $\mathcal{L}(X)$ associated to $X$ plays a central
role in the theory of motivic integration \cite{DLinvent}. The aim
of this section is to clarify its relation to the $k$-analytic
space $X^{an}$. For the definition of the arc scheme
$\mathcal{L}(X)$, we refer to \cite[\S\,1]{DLinvent}. We recall
that it is a separated $k$-scheme, generally not Noetherian, which
is constructed as a projective limit of the separated $k$-schemes
of finite type $\mathcal{L}_n(X)$, $n\geq 0$, where
$\mathcal{L}_n(X)$ is the Weil restriction of $X\times_k
k[t]/(t^{n+1})$ to $\Spec k$. We denote by
$\pi_n:\mathcal{L}(X)\rightarrow \mathcal{L}_n(X)$ the natural
projection, for each $n\geq 0$. Note that $\mathcal{L}_0(X)=X$.

For any field $k'$ containing $k$, there is a natural bijection
between $\mathcal{L}(X)(k')$ and $X(k'[[t]])$. An element $\psi$
of this set is called a $k'$-valued arc on $X$. The image under
$\psi:\mathrm{Spec}\,k'[[t]]\rightarrow X$ of the closed point of
$\mathrm{Spec}\,k'[[t]]$ is called the origin of the arc; it
coincides with the image of $\pi_0\circ \psi\in X(k')$. The image
of the generic point of $\mathrm{Spec}\,k'[[t]]$ is called the
generic point of the arc.

\begin{prop}
Let $X$ be a separated $k$-scheme of finite type, and denote by
$\mathcal{L}(X)$ its arc space. There exists a natural map of sets
$\Theta:\mathcal{L}(X)\times [0,1[\,\rightarrow X_\eta$, such that

(1) the composition $i_X\circ \Theta:\mathcal{L}(X)\rightarrow X$
maps a couple $(\psi,r)$ to the generic point of $\psi$ if $r>0$,
and to the origin of $\psi$ if $r=0$,

(2) the map $\Theta(\cdot,0):\mathcal{L}(X)\rightarrow X_\eta$
coincides with the composition $s_X\circ \pi_0$.

(3) for any $r\in [0,1[$ and any arc $\psi$ in $\mathcal{L}(X)$,
$(sp_X\circ \Theta)(\psi,r)$ is the origin of $\psi$.
\end{prop}

\begin{proof}
We may assume that $X$ is affine, say $X=\mathrm{Spec}\,A$. For
any field $k'$ containing $k$, a $k'$-rational point $\psi$ on
$\mathcal{L}(X)$ corresponds canonically to a morphism of
$k$-algebras $\psi:A\rightarrow k'[[t]]$. We define
$\Theta(\psi,r)$ as follows. If $0<r<1$, then we put
$\Theta(\psi,r)(a):=r^{v_t(\psi(a))}$ for each $a\in A$, where
$v_t$ denotes the $t$-adic valuation $k'[[t]]\rightarrow
\N\cup\{\infty\}$. If $r=0$, then we put $\Theta(\psi,r)(a):=1$ if
$v_t(\psi(a))= 0$, and $\Theta(\psi,r)(a):=0$ else. Note that
$v_t(\psi(a))\neq 0$ iff $a(x)=0$, with $x$ the origin of $\psi$.
%
%Now we show that $f$ is continuous. By definition of the topology
%on $f$, it suffices to show that $(\psi,r)\mapsto f(\psi,r)(a)$ is
%continuous for each $a$ in $A$. This follows immediately from the
%fact that $\psi\mapsto v_t(\psi(a))$ is continuous on
%$\mathcal{L}(X)$, for each $a\in A$.
%The remainder of the statement is clear.
\end{proof}

We emphasize that the map $\Theta$ is not continuous. If $x$ is a
closed point of $X$, then $sp^{-1}_X(x)$ is an open subset of
$X_\eta$, while $(sp_X\circ \Theta)^{-1}(x)$ consists of the
couples $(\psi,r)$ with $r\in [0,1[$ and $\psi$ a point of
$\mathcal{L}(X)$ with origin $x$. This is not an open subset of
$\mathcal{L}(X)\times[0,1[$ in general. The map $\Theta$ is not
anti-continuous, either, since
$\Theta^{-1}(s_X(x))=\{\psi_x\}\times [0,1[$, with $\psi_x$ the
constant arc at $x$. However, it is easily seen that $\Theta$
becomes continuous if we endow $\mathcal{L}(X)$ with the $t$-adic
topology (i.e. the limit topology w.r.t. the discrete topology on
the spaces $\mathcal{L}_n(X)$).

The map $\Theta$ is not injective on any of the fibers of the
projection $\mathcal{L}(X)\times [0,1[\,\rightarrow [0,1[$. For
instance, if $X$ is irreducible, then all couples $(\psi,r)$ for
which the origin of $\psi$ is the generic point $\eta_X$ of $X$,
are mapped to the same point of $X_\eta$, namely to the point
$s_X(\eta_X)$. Moreover, if
$\psi:\mathrm{Spec}\,k'[[t]]\rightarrow X$ is an arc on $X$, and
$\theta$ is a $k$-automorphism of $\mathrm{Spec}\,k'[[t]]$, then
it follows immediately from the construction that
$\Theta(\psi,r)=\Theta(\psi\circ \theta,r)$ for any $r\in [0,1[$.

Things look a little better if we assume that $k$ is algebraically
closed, of characteristic zero, and if we only consider points in
$\mathcal{L}(X)(k)$.
 Restricting $\Theta$, we obtain a map
$\Theta_{k}:\mathcal{L}(X)(k)\times [0,1[\,\rightarrow X_\eta$. It
is still not continuous w.r.t. the Zariski topology on
$\mathcal{L}(X)(k)$ (by the same example as above) but we have the
following result.

\begin{prop}\label{fibers}
Assume that $k$ is algebraically closed, of characteristic zero.
If $r\in \,]0,1[$ and $\psi,\,\varphi\in \mathcal{L}(X)(k)$, then
$\Theta(\psi,r)=\Theta(\varphi,r)$ iff there exists a
$k$-automorphism $\theta$ of $\Spec k[[t]]$ such that
$\psi=\varphi\circ \theta$ (if we view $\psi$ and $\varphi$ as
arcs $\Spec k[[t]]\rightarrow X$).

Moreover, if we endow $\mathcal{L}(X)(k)$ with the $t$-adic
topology, then the map
$$\Theta_{k}:\mathcal{L}(X)(k)\times
[0,1[\,\rightarrow X_\eta$$ is continuous.
\end{prop}
\begin{proof}
We may assume that $X$ is affine, say $X=\mathrm{Spec}\,A$.  To
prove the statement about the fibers of $\Theta_{k}(r,\cdot)$, it
suffices to show that two morphisms of $k$-algebras
$\varphi,\,\psi:A\rightarrow k[[t]]$ differ by a $k$-automorphism
of $k[[t]]$ if $v_t\circ \varphi=v_t\circ \psi$ (with $v_t$ the
$t$-adic valuation on $k[[t]]$).
%We may assume that $A$ is a
%polynomial algebra $k[x_1,\ldots,x_n]$, and that
%$$q:=v_t(\varphi(x_1))=v_t(\psi(x_1))\leq
%v_t(\varphi(x_i))=v_t(\psi(x_i))$$ for $i=1,\ldots,n$ (by
%convention, $v_t(0)=\infty\geq j$ for all $j\in \N$).
%Composing
%$\varphi$ and $\psi$ with $k$-automorphisms of $k[[t]]$, we may as
%well suppose that $\varphi(x_1)=\psi(x_1)=t^q$. We will show that,
%in this case, $\varphi=\psi$.

The fact that for all $a\in A$, $(v_t\circ \varphi)(a)=\infty$ iff
$(v_t\circ \psi)(a)=\infty$, implies that the kernels of $\varphi$
and $\psi$ coincide; we denote this prime ideal of $A$ by $I$, and
we denote by $F$ the quotient field of $A/I$. Then $\varphi$ and
$\psi$ extend uniquely to morphisms of $k$-fields
$\varphi',\,\psi':F\rightarrow k((t))$ such that
$$v:=v_t\circ\varphi'=v_t\circ \psi'$$

There exists a unique integer $\gamma\geq 0$ such that
$v(F^*)=\gamma\cdot \Z$. The equality $\gamma=0$ can only occur if
$A/I=F=k$, so we may exclude this case. Let $x$ be an element of
$F$ such that $v(x)=\gamma$. Composing $\varphi$ with a
$k$-automorphisms of $k[[t]]$ which multiplies $t$ with a
$\gamma$-th root of $t^{-\gamma}\varphi(x)$, we may assume that
$\varphi(x)=t^\gamma$, and, likewise, that $\psi(x)=t^{\gamma}$.

Let $y$ be any element of $A/I$, and  write $\varphi(y)$ as
$\sum_{i\geq i_0}b_i t^i$ with $b_i\in k$ and $b_{i_0}\neq 0$. We
know that $i_0=v(y)$ is divisible by $\gamma$. Moreover, the fact
that
$$(v_t\circ \psi)(y-b_{i_0}x^{i_0/\gamma})=(v_t\circ
\varphi)(y-b_{i_0}x^{i_0/\gamma})=v_t(\sum_{i>i_0}b_it^i)>i_0$$
implies that $\psi(y)=\varphi(y)$ modulo $t^{i_0+1}$. Repeating
the argument with $y$ replaced by $y-b_{i_0}x^{i_0/\gamma}$ we can
conclude that $\psi(y)=\varphi(y)$.

Continuity of $\Theta_k$ follows, of course, from continuity of
$\Theta$.
\end{proof}
\begin{remark}
Proposition \ref{fibers} is false if we do not assume that $k$ is
algebraically closed: consider the $\Q$-morphisms
$\varphi,\,\psi:\Q[x]\rightarrow \Q[[t]]$ mapping $x$ to $t^2$,
resp $2t^2$. The condition that $k$ has characteristic zero is
also necessary: if $k$ is an algebraically closed field of
characteristic $p>0$, then the $k$-morphisms
$\varphi,\,\psi:k[x]\rightarrow k[[t]]$ mapping $x$ to $t^p$, resp
$t^p+t^{p+1}$ do not differ by a $k$-automorphism of $k[[t]]$,
even though they define the same $t$-adic valuation on $k[x]$.
\end{remark}

To conclude, we look at $\Theta_{k}$ from a different point of
view. Assume that $k$ is algebraically closed, of arbitrary
characteristic. Consider the $stft$ formal $R$-scheme
$\X=X\widehat{\times}_k\Spf R$. Recall that $X_\eta\cong
\X(0)_\eta$.

There are natural bijections
$$\mathcal{L}(X)(k)=X(R)=\X(R)$$ so any point
$\psi$ of $\mathcal{L}(X)(k)$ gives rise to a morphism of formal
$R$-schemes $\Spf R\rightarrow \X$ and hence to a morphism of
$k$-analytic spaces $\psi^k_\eta:\mathfrak{R}_\eta\rightarrow
\X^k_\eta$ (we denote by $\mathfrak{R}$ the special formal
$k$-scheme $\Spf R$, as before). Using the natural homeomorphism
$\mathfrak{R}_\eta\cong [0,1[$ from Lemma \ref{dischomeo}, we get
a canonical map
$$\Upsilon:\mathcal{L}(X)(k)\times [0,1[\,\rightarrow
\X^k_\eta:(\psi,r)\mapsto \psi^k_\eta(r)$$  which satisfies
$(\lambda\circ \Upsilon) (\cdot,r)=r$ for all $r\in [0,1[$ (here
$\lambda:\X^k_\eta\rightarrow [0,1[$ is the natural map of
topological spaces from Section \ref{sec-hom}).
%The maps
%$\Upsilon(\cdot,0)$ and $\Theta_{k}(\cdot,0)$ coincide (they are
%both equal to the composition $s_X\circ
%\pi_0:\mathcal{L}(X)(k)\rightarrow X_\eta(k)$).

\begin{prop} If we endow $\mathcal{L}(X)(k)$ with the $t$-adic
topology, then $\Upsilon$ is continuous, and
 the resticted map
$$\Upsilon':\mathcal{L}(X)(k)\times\,]0,1[\,\rightarrow
\lambda^{-1}(\,]0,1[\,)$$ is a homeomorphism onto
$\cup_{r\in\,]0,1[}\X(r)_\eta(K_r)$ (with the topology induced
from the one on $\X^k_\eta$).
\end{prop}
\begin{proof}
To prove continuity of $\Upsilon$ we may assume that $X$ is
affine, say $X=\Spec A$. Since $\X^k_\eta$ is a subspace of the
affinoid space $\mathscr{M}(A[t])$ (where $A[t]$ carries the
trivial Banach norm) it suffices to show that, for any $f\in
A[t]$, the map
$$\mathcal{L}(X)\times [0,1[\,\rightarrow \R:(\psi,r)\mapsto
|f(\psi^k_\eta(r))|$$ is continuous. Viewing $\psi$ as a
$k$-morphism $A\rightarrow k[[t]]$, the right hand side is nothing
but $p_r(\psi^*f)$ where $\psi^*f$ is the image of $f$ under the
morphism of $A$-algebras $A[t]\rightarrow k[[t]]$ which sends $t$
to itself, and $p_r(\sum_{i\geq 0}a_i t^i)=\max |a_i|_0 r^i$ with
the convention that $0^0=1$. In other words,
$$|f(\psi^k_\eta(r))|=r^{v_t(\psi^*f)}$$ which is clearly
continuous as a function in $(\psi,r)$.

 For each $r\in\,]0,1[$ the canonical bijection
$\X(R)=\X(r)_\eta(K_r)$ defines a bijection
$\beta_r:\mathcal{L}(X)(k)\rightarrow \X(r)_\eta(K_r)$, and one
verifies in a straightforward way that this is a homeomorphism.
Taking the product with the identity map on $]0,1[$ we get a
homeomorphism
$$\beta_r\times id:\mathcal{L}(X)(k)\times\,]0,1[\,\rightarrow \X(r)_\eta(K_r)\times \,]0,1[$$
whose composition with the homeomorphism $\phi$ from Proposition
\ref{fiberhomeo} coincides with $\Upsilon'$.
\end{proof}

Since taking generic fibers commutes with fiber products, there
exists a natural isomorphism of $k$-analytic spaces
$\X^k_\eta\cong X_\eta\times_k \mathfrak{R}_\eta$. The projection
to the first factor defines a morphism $p_1:\X^k_\eta\rightarrow
X_\eta$, and $\Theta_k=p_1\circ \Upsilon$.
%\begin{prop}
%There exists a natural homeomorphism
%$$\Xi:X_\eta\times [0,1[\,\rightarrow \X^k_\eta$$  which restricts to a strong deformation retract
%$$\left(\bigcup_{r\in\,[0,1[}\X(r)_\eta(K_r)\right)\times [0,1]\rightarrow \bigcup_{r\in\,[0,1[}\X(r)_\eta(K_r)$$ onto
%$\X(0)_\eta(k)=X_\eta(k)$, and such that
%$\Theta_{alg}=\Xi(\cdot,1)\circ \Upsilon$.
%\end{prop}
%\begin{proof}
%We may assume that $X$ is affine, say $X=\Spec A$.
%\end{proof}
% (see
%\cite[9.1.2]{NiSe}; we denote by $\X(r)_{rig}$ the rigid generic
%fiber of $\X(r)$ in the sense of \cite{Raynaud}; it is naturally
%isomorphic to the rigid $K_r$-variety associated to $\X(r)_\eta$
%\cite[3.3]{Berk1}).
%% If $\psi$ is a point of $\mathcal{L}(X)(k)=X(k[[t]])=\X(k[[t]])$,
%% then $\beta_r(\psi)$ is the point $\psi_\eta:\mathc

%It is not hard to see that $\beta_r$ is a homeomorphism if we endow
%both sides with the $t$-adic topology; recall that the $t$-adic
%topology on $\X(r)_{rig}$ coincides with the topology induced by
%the spectral topology on $\X(r)_\eta$ via the natural injection
%$\X(r)_{rig}\rightarrow \X(r)_\eta$.
%
%Let $$\beta:\mathcal{L}(X)_{alg}\times [0,1[\,\rightarrow
%\X^k_\eta$$ be the map defined by
%$\beta(\cdot,r)=\beta_{r}(\cdot)$ for $0<r<1$ and
%$\beta(\cdot,0)=\Theta_{alg}(\cdot,0)$; here we identified
%$\X(r)_\eta$ with the fiber over $r$ of the natural map
%$\lambda:\X^k_\eta\rightarrow [0,1[$ (see Proposition
%\ref{fiberhomeo}) for each $r\in [0,1[$.
\fi
\section{Weight zero part of the mixed Hodge structure on the
nearby cohomology}\label{sec-mhs}
\subsection{Cocubical systems}
We recall the following definition. For any finite, non-empty set
$S$, we denote by $\square_S$ the set of \textit{non-empty}
subsets of $S$, ordered by inclusion. For any category
$\mathcal{C}$, the category of $S$-cocubical systems in
$\mathcal{C}$ is the category of covariant functors
$\square_S\rightarrow \mathcal{C}$ (with natural transformations
as morphisms).

Let $\mathcal{A}$ be an abelian category, and denote by
$C^+(\mathcal{A})$ the category of bounded below complexes in
$\mathcal{A}$. We fix an integer $n\geq 0$, and we consider a
$[n]$-cocubical system $(C^\bullet_L)_{L\in \square_{[n]}}$ in
$C^+(\mathcal{A})$. For each object $L$ in $\square_{[n]}$ and
each $p\in \Z$, we denote by $d_L:C^p_L\rightarrow C^{p+1}_L$ the
differential in the complex $C^\bullet_L\in C^+(\mathcal{A})$. For
each couple $(L,L')$ of objects in $\square_{[n]}$ with $L\subset
L'$, we denote by
$$\delta^{L}_{L'}:C^\bullet_{L}\rightarrow C^{\bullet}_{L'}$$ the
face map which is part of the cocubical system.
 We define the associated
simple complex $s_{\square}((C_L^\bullet)_{L\in \square_{[n]}})$
as follows (we use the notation $s_\square$ to avoid confusion
with the simple complex $s(\cdot)$ associated to a double
complex).

First, we define a double complex $A^{\bullet\bullet}$. We put
$$A^{p,q}=\left\{ \begin{array}{ll}0&\mbox{ for }(p,q)\in \Z\times \Z_{<0}
\\ \oplus_{|L|=q+1} C^p_L&\mbox{ for }(p,q)\in \Z\times
\Z_{\geq 0}\end{array}\right.$$ The horizontal differential
$A^{p,q}\rightarrow A^{p+1,q}$, for $q\geq 0$, is given by
$$\oplus_{|L|=q+1}\{d_L:C^p_L\rightarrow C^{p+1}_L\} $$
The restriction of the vertical differential $A^{p,q}\rightarrow
A^{p,q+1}$ to the component $C^p_L$ is given by $$C^p_L\rightarrow
A^{p,q+1}:\sum_{i\in [n]\setminus
L}(-1)^{\varepsilon(L,i)+p+1}\delta_{L\cup\{i\}}^L$$ where
$\varepsilon(L,i)$ denotes the number of elements $j$ in $L$ with
$j<i$.

 We define $s_{\square}((C^\bullet_L)_{L\in \square_{[n]}})$ as the associated simple complex
$s(A^{\bullet \bullet})$; we'll also denote it by
$s_{\square}(C^\bullet_L)$ to simplify notation. A map
$f_L:C^\bullet_L\rightarrow D^\bullet_L$ of $[n]$-cocubical
systems in $C^+(\mathcal{A})$ induces a map
$s_{\square}(f_L):s_{\square}(C^\bullet_L)\rightarrow
s_{\square}(D^\bullet_L)$ between the associated simple complexes,
so we get a functor $s_{\square}(\cdot)$ from the category of
$[n]$-cocubical systems in $C^+(\mathcal{A})$ to the category
$C^+(\mathcal{A})$.

Denote by $C^+(\mathcal{A},W,F)$ the category of bifiltered
bounded below complexes in $\mathcal{A}$ (with $F$ decreasing and
$W$ increasing). If $(C_L^\bullet,W,F)_{L\in \square_{[n]}}$ is a
$[n]$-cocubical system
 in $C^+(\mathcal{A},W,F)$ then
we endow the simple complex $C^\bullet=s_{\square}(C^\bullet_L)$
with the following filtrations: for each $n,\,r\in \Z$, we put
\begin{eqnarray*}
F^rC^n&=&\oplus_{p+q=n,\,q\geq 0}\oplus_{|L|=q+1}F^rC^p_L
\\ W_r C^n&=&\oplus_{p+q=n,\,q\geq 0}\oplus_{|L|=q+1}W_{r+q}C^p_L
\end{eqnarray*}
We call the bifiltered complex $(s_{\square}(C^\bullet_L),W,F)$
the associated simple complex of the cocubical system
$(C_L^\bullet,W,F)_{L\in \square_{[n]}}$, and we put
$$s_{\square}(C_L^\bullet,W,F)=(s_{\square}(C^\bullet_L),W,F)$$
This defines a functor $s_{\square}$ from the category of
$[n]$-cocubical systems in $C^+(\mathcal{A},W,F)$ to the category
$C^+(\mathcal{A},W,F)$.
%Note that, in this definition,
%$C_\emptyset^\bullet$ does not play a role.
%
%We'll denote $C^\bullet$ by $s((C_L)_{L\in \square_{[n]}})$.

Let us consider some elementary examples, which will be of use
later on.
\begin{example}\label{nc}
Let $X$ be a smooth complex variety and let $E$ be a proper strict
normal crossing divisor on $X$, with irreducible components
$E_i,\,i=0,\ldots,n$. For any $L\in \square_{[n]}$, we put
$E_L=\cap_{i\in L}E_i$ and denote by $a_L:E_L\rightarrow X$ the
inclusion. Consider the bifiltered complex
$\mathcal{H}_L=((a_L)_*\Omega^\bullet_{E_L},W,F)$ where  $F$ is
the stupid filtration, and $W$ is given by
$$ W_i((a_L)_*\Omega_{E_L}^{\bullet})=\left\{\begin{array}{l} 0 \mbox{ for }
i<0
\\ {}
\\ (a_L)_*\Omega_{E_L}^{\bullet}
\mbox{ for } i\geq 0
\end{array}\right.$$
In other words, $\mathcal{H}_L$ is the image under $(a_L)_*$ of
the complex component $Hdg^\bullet(E_L)_\C$ of the mixed Hodge
complex of sheaves associated to the smooth and proper complex
variety $E_L$. If $L'$ is another subset of $[n]$ and $L\subset
L'$, then the closed immersion $E_{L'}\rightarrow E_L$ induces
restriction maps $\mathcal{H}_L\rightarrow \mathcal{H}_{L'}$,
which make $(\mathcal{H}_L)_{L\in \square_{[n]}}$ into a
$[n]$-cocubical system. We denote its associated simple complex by
$Hdg^\bullet(E)_{\C}$. It is the complex component
 of a mixed Hodge complex of sheaves which induces the canonical
 mixed Hodge structure on $H^*(E,\Z)$
 \cite[3.5]{steenbrink-limit}.
\end{example}

\begin{example}\label{pieces}
Let $X$ be any topological space, and consider a cover
$\{E_i\,|\,i\in [n]\}$ of $X$ by closed subsets. For any $L\in
\square_{[n]}$, we put $E_L=\cap_{i\in L}E_i$ and denote by
$a_L:E_L\rightarrow X$ the inclusion.

 Let $G^\bullet$
be any object in $C^+(X)$ (i.e. a bounded below complex of abelian
sheaves on $X$), and consider the $[n]$-cocubical system defined
by $G^\bullet_L=(a_L)_*(a_L)^*G^\bullet$. By adjunction, we have a
natural map of complexes
$$G^\bullet\rightarrow \oplus_{i\in [n]}(a_i)_*(a_i)^*G^\bullet$$
where we wrote $a_i$ instead of $a_{\{i\}}$. The target of this
map is a direct summand of $s_{\square}(G^\bullet_L)$, so we get a
map of complexes $G^\bullet\rightarrow s_{\square}(G^\bullet_L)$.
This is a quasi-isomorphism, since for each $q\geq 0$, the $q$-th
row of the double complex $A^{\bullet \bullet}$ associated to the
cocubical system $G^\bullet_L$ is a resolution of $G^q$. If
$f:G^\bullet \rightarrow H^\bullet$ is a morphism in $C^+(X)$,
then $f$ induces a morphism of cocubical systems
$f_L:G_L^\bullet\rightarrow H^\bullet_L$, and the square
$$\begin{CD}
G^\bullet @>f>> H^\bullet \\@VVV @VVV
\\s_{\square}(G^\bullet_L)@>s_{\square}(f_L)>> s_{\square}(H^\bullet_L)
\end{CD}$$
commutes.
%
%If $G^\bullet$ is endowed with filtrations $W$ and $F$, then we
%can define filtrations on each component $G^\bullet_L$ by
%$F^rG^\bullet_L=(a_L)_*(a_L)^*(F^rG\bullet)$ and
%$W_rG^\bullet_L=(a_L)_*(a_L)^*(W_{r+1-|L|}G\bullet)$. In this way,
%$$(G^\bullet,W,F)\rightarrow (tot(G^\bullet_L),W,F)$$ becomes a
%bifiltered quasi-isomorphism.
\end{example}

\begin{example}\label{diffpieces}
We keep the notations of Example \ref{pieces}, supposing moreover
that $X$ is a smooth complex variety, and that $E_L$ is a smooth
closed subvariety for each $L\in \square_{[n]}$. We consider the
$[n]$-cocubical system of complexes
$$((a_L)_*\Omega^\bullet_{E_L})_{L\in \square_{[n]}}$$ in
$C^+(X,\C)$.
 The product of restriction maps
$\Omega^\bullet_X\rightarrow \oplus_{i\in
[n]}(a_i)_*\Omega^\bullet_{E_i}$ induces a map of complexes
$\Omega_X^\bullet\rightarrow
s_{\square}((a_L)_*\Omega^\bullet_{E_L})$. This map is a
quasi-isomorphism: via the quasi-isomorphisms
$\Omega^\bullet_X\cong \C_X$ and $\Omega^\bullet_X\cong \C_{E_L}$
and the exactness of the functor
% $(a_L)^*:C^+(X,\C)\rightarrow
%C^+(E_L,\C)$ and
$(a_L)_*:C^+(E_L,\C)\rightarrow C^+(X,\C)$, we
recover the situation of Example \ref{pieces}, with
$G^\bullet=\C_X$.
\end{example}

\subsection{Localized limit mixed Hodge complex}\label{sec-loclim}
Let $S$ be the open complex unit disc, and let $f:X\rightarrow S$
be a projective morphism, with $X$ a complex manifold. We fix a
complex coordinate $t$ on $S$. Assume that the special fiber $X_s$
is a reduced strict normal crossing divisor $E=\sum_{i\in I}E_i$.
We fix a total order on $I$, i.e. a bijection $I\cong [a]$ for
some integer $a\geq 0$. Let $J$ be a non-empty subset of $I$, put
$E(J)=\cup_{i\in J}E_i$, and denote by $v_J:E(J)\rightarrow X_s$
the inclusion map. The total order on $I$ induces a total order on
$J$, i.e. a bijection $J\cong [b]$ for some integer $b\geq 0$.
These total orders are necessary to apply the functor
$s_{\square}$ to $I$- and $J$-cocubical systems.

In \cite{navarro-invent} (see also \cite{navarro-survey}), Navarro
Aznar constructed a localized limit integral mixed Hodge complex
of sheaves on $E(J)$, whose integral component is quasi-isomorphic
to $v_J^*R\psi_{f}(\Z)$ (here $R\psi_f$ is the complex analytic
nearby cycle functor associated to $f$). This mixed Hodge complex
induces a canonical integral mixed Hodge structure on the
hypercohomology spaces
$$\mathbb{H}^i(E(J),v_J^* R\psi_{f}(\Z))$$

We will follow the approach in \cite[\S\,12.1]{peters-steenbrink}.
We briefly recall the definition of the complex component
$v_J^*\psi^H_f(\C)$ of the localized limit integral mixed Hodge
complex of sheaves $v_J^*\psi^{H}_f$ on $E(J)$, in order to fix
notations (the notations we adopt here differ slightly from the
ones in \cite[\S\,12.1]{peters-steenbrink}).

For any non-empty subset $L$ of $J$, we denote by
$\mathcal{I}_{E_L}$ the defining ideal sheaf of $E_L=\cap_{i\in
L}E_i$ in $X$. Then $\mathcal{I}_{E_L}\Omega^\bullet_{X}(\log E)$
is a subcomplex of $\Omega^\bullet_X(\log E)$, and we consider the
bifiltered complex
$$\Omega^\bullet_{X}(\log E)|_{E_L}=\Omega^\bullet_X(\log
E)/\mathcal{I}_{E_L}\Omega^\bullet_X(\log E)$$ in $C^+(E,\C)$,
endowed with the quotient filtrations $W$ and $F$ (the quotient of
the weight filtration, resp. of the stupid filtration on
$\Omega^\bullet_{X}(\log E)$). We define a double complex
$A^{\bullet \bullet}_L$ by
$$A^{p,q}_L=\Omega^{p+q+1}_{X}(\log E)|_{E_L}/W_p(\Omega^{p+q+1}_{X}(\log E)|_{E_L})$$ for $p,q\geq
0$ and
 we consider $A^{\bullet\bullet}_L$ as a sheaf on $E(J)$.
The
differentials $d'_L$ and $d''_L$ are given by
\begin{eqnarray*}
d'_L&:&A^{p,q}_L\rightarrow A^{p+1,q}_L:\omega \mapsto
(dt/t)\wedge\omega
\\ d''_L&:& A^{p,q}_L\rightarrow A^{p,q+1}_L:\omega \mapsto d\omega
\end{eqnarray*}
We put an increasing filtration $W(M)$ (the \textit{monodromy
weight filtration}) on $A^{\bullet\bullet}_L$ by defining
$W(M)_rA^{p,q}_L$ as the image of $W_{r+2p+1}\Omega^{p+q+1}_X(\log
E)$ in $A^{p,q}_L$; this induces a filtration $W(M)$ on the
associated simple complex $C^\bullet_L=s(A^{\bullet \bullet}_L)$.
We also endow $C^\bullet_L$ with a decreasing filtration $F$ by
putting $F^rC^n_L=\oplus_{p+q=n,q\geq r}A_L^{p,q}$.

%We fix a total ordering on $J$, i.e. a bijection $J\cong [a]$ for
%some $a\geq 0$.
When $L$ varies over the non-empty subsets of $J$,
the bifiltered complexes $(C_L^\bullet,W(M),F)$ form a
$J$-cocubical complex of bifiltered objects in $C^+(E(J),\C)$,
whose associated simple complex $s_{\square}(C_L^\bullet,W(M),F)$
is denoted by $v_J^*\psi^H_f(\C)$. The complex
$s_{\square}(C_L^\bullet)$ is quasi-isomorphic to
$v_J^*R\psi_f(\C)$.

\subsection{The specialization map on the mixed Hodge level}
We keep the notations of Section \ref{sec-loclim}. We would like
to lift the natural specialization map $\C_{E(J)}\rightarrow
v_J^*R\psi_f(\C)$ to a map of bifiltered complexes
$$Hdg^\bullet(E(J))_{\C}\rightarrow v_J^*\psi_f^H(\C)$$ (here
$Hdg^\bullet(E(J))_{\C}$ is the complex introduced in Example
\ref{nc}). In the global case $J=I$, this was done in
\cite[\S\,11.3.1]{peters-steenbrink}.

For any non-empty subset $M$ of $I$, we denote by $a_M$ the closed
immersion $E_M\rightarrow E$. In the proof of
\cite[Thm.\,6.12]{peters-steenbrink}, the
 following property was shown.
\begin{lemma}\label{res}
There is a commutative diagram
 $$\begin{CD}
  Gr^W_m(\Omega^\bullet_X(\log
E))@>>\cong > \bigoplus_{\stackrel{M\subset
I}{|M|=m}}(a_M)_*\Omega^\bullet_{E_{M}}[-m]
\\ @VVV @VVV
\\ Gr^W_m(\Omega^\bullet_X(\log E)|_{E_L})@>\cong >> \bigoplus_{\stackrel{M\subset
I}{|M|=m}} (a_{L\cup M})_*\Omega^\bullet_{E_{L\cup
M}}[-m]\end{CD}$$ for any couple non-empty subset $L$ of $I$ and
any integer $m\geq 1$. Here the upper horizontal map is the
isomorphism induced by the residue map, the left vertical arrow
comes from the natural projection $\Omega^\bullet_X(\log
E)\rightarrow \Omega^\bullet_X(\log E)|_{E_L}$, and the right
vertical arrow is the obvious restriction map.
\end{lemma}

By Example \ref{diffpieces} (applied to $X=E_L$ and the cover
$\{E_{L\cup i}\,|\,i\in I\}$), we have for each non-empty subset
$L$ of $J$ a quasi-isomorphism
$$(a_L)_*\Omega^\bullet_{E_L}\cong \oplus_{p+q=\bullet,\,p\geq 0}\oplus_{M\subset I,|M|=p+1}(a_{L\cup M})_*\Omega^q_{E_{L\cup
M}} $$ and using the isomorphism in Lemma \ref{res}, we get a
natural quasi-isomorphism
$$(a_L)_*\Omega^\bullet_{E_L}\rightarrow \oplus_{p+q=\bullet,\,p\geq 0} Gr^W_{p+1}(\Omega^{p+q+1}_X(\log E)|_{E_L})$$
whence a morphism of $J$-cocubical systems of bifiltered complexes
\begin{equation}\label{s_L}
\sigma_L:((a_L)_*\Omega^\bullet_{E_L},W,F)\rightarrow
(C_L^\bullet,W(M),F)\end{equation}
 where the left hand side is
defined as in Example \ref{nc}.
 Passing to the
associated simple complexes, we get a map of bifiltered complexes
\begin{equation}\label{s}
\sigma:Hdg^\bullet(E(J))_{\C}\rightarrow
v_J^*\psi^H_f(\C)\end{equation}

% For each $L$, consider the composition
%$$(a_L)_*\Omega_{E_L}^{\bullet}[-|L|]\rightarrow
%Gr^W_{|L|}\Omega^\bullet_{X}(\log E)\rightarrow
%Gr^W_{|L|}(\Omega^{\bullet}_{X}(\log E)|_{E_L})\rightarrow
%^{|L|-1,\bullet-|L|}_L$$ where $a_L$ is the inclusion
%E_L\rightarrow E(J)$ and the first arrow comes from the residue
%map. This composition induces a morphism of complexes
%$s_L:(a_L)_*\Omega_{E_L}^{\bullet}[-|L|+1]\rightarrow
%C^{\bullet}_L$. If we endow
%$(a_L)_*\Omega_{E_L}^{\bullet}[-|L|+1]$ with the stupid filtration
%$F$ and the weight filtration $W$ given by
%$$ W_i((a_L)_*\Omega_{E_L}^{\bullet}[-|L|+1])=\left\{\begin{array}{l} 0 \mbox{ for }
%i<0
%\\ {}
%\\ (a_L)_*\Omega_{E_L}^{\bullet}[-|L|+1]
%\mbox{ for } i\geq 0
%\end{array}\right.$$
%then $s_L$ becomes a morphism of bifiltered complexes
%$$s_L:((a_L)_*\Omega_{E_L}^{\bullet}[-|L|+1],W,F)\rightarrow
%(C_L^\bullet,W(M),F)$$ When $L$ varies over the non-empty subsets
%of $J$, then $((a_L)_*\Omega_{E_L}^{\bullet}[-|L|+1],W,F)$ defines
%a $J$-cocubical complex (with $\emptyset$-component $0$), and its
%total complex is nothing but the complex component
%$Hdg^{\bullet}(E(J))_{\C}$ of the mixed Hodge complex of sheaves
%associated to the strict normal crossing divisor $E(J)$. The maps
%$s_L$ define a morphism of bifiltered complexes
%\begin{equation}\label{s}
%s:Hdg^\bullet(E(J))_{\C}\rightarrow
%v_J^*\psi^H_f(\C)\end{equation}
\begin{lemma}
The square
$$\begin{CD}
\C_{E(J)}@>spec>> v_J^*R\psi_f(\C)
\\ @VVV @VVV
\\ Hdg^\bullet(E(J))_{\C}@>\sigma >>v_J^*\psi_f^H(\C)
\end{CD}$$ commutes in $D^+(E(J),\C)$. The upper horizontal map is
the natural specialization map, and the vertical maps are the
natural comparison isomorphisms in $D^+(E(J),\C)$ which are part
of the mixed Hodge complexes $Hdg^\bullet (E(J))$, resp.
$v_J^*\psi_f^H$.
%The map
%$$\mathbb{H}^i(E(J),Hdg^\bullet(E(J))_\C)\rightarrow
%\mathbb{H}^i(E(J),v_J^*\psi_f^H(\C))$$ induced by $\sigma$
%coincides with the map $$H^i(E(J),\C)\rightarrow
%\mathbb{H}^i(E(J),v_J^*R\psi_f(\C))$$ induced by the natural
%specialization map $\C_{E}\rightarrow R\psi_{f}(\C)$.
\end{lemma}
\begin{proof}
%It will be more convenient to prove this statement on the level of
%rational components $Hdg^\bullet(E(J))_\Q$ and $v_J^*\psi^H_f(\Q)$
%of the respective mixed Hodge complexes of sheaves. The first is
%the total complex associated to the $J$-cocubical syetem
%$(\Q_{E_L})_{L\subset J}$, the second is defined in
%\cite[12.1.1]{peters-steenbrink}.
In the global case $I=J$, Peters and Steenbrink defined a limit
integral mixed Hodge complex of sheaves $\psi^{H}_f$ on $E$ (see
\cite[11.2.7]{peters-steenbrink}); we denote by $\psi^H_f(\C)$ its
complex component. The complex $\psi^H_f(\C)$ is defined as the
simple complex associated to the double complex
$$A^{p,q}=\Omega^{p+q+1}(\log E)/W_p\Omega^{p+q+1}(\log E)$$
where the differentials are defined in the same way as for
$A^{\bullet \bullet}_L$. By \cite[11.3.1]{peters-steenbrink}, we
have the following commutative diagram in $D^+(E,\C)$\,:
$$\begin{CD}
\C_E@>spec>> R\psi_f(\C)
\\ @VVV @VVV
\\Hdg^\bullet(E)_{\C}@>\beta>> \psi^H_f(\C)
\end{CD}$$
%where $spec$ is the specialization map.
 The vertical maps are the natural comparison isomorphisms, and
$\beta$ is constructed by means of the residue maps:
$$\bigoplus_{|L|=p+1}(a_L)_*\Omega^q_{E_L}\cong Gr^W_{p+1}\Omega^{p+q+1}(\log
E)\rightarrow A^{p,q}=\Omega^{p+q+1}(\log
E)/W_p(\Omega^{p+q+1}(\log E))$$
 Applying the functor $a_L^*:D^+(E,\C)\rightarrow D^+(E_L,\C)$, for any non-empty subset $L$ of
$J$, we get a commutative diagram in $D^+(E_L,\C)$\,:
$$\begin{CD}
\C_{E_L}@>spec>> a_L^*R\psi_f(\C)
\\ @VVV @VVV
\\a_L^*(Hdg^\bullet(E)_{\C})@>\beta_L>> a_L^*\psi^H_f(\C)
\end{CD}$$
We denote by $b_L$ the closed immersion $E_L\rightarrow E(J)$, and
we define an isomorphism
$\gamma_L:a_L^*(Hdg^\bullet(E)_{\C})\rightarrow
\Omega^\bullet_{E_L}$ in $D^+(E_L,\C)$ as the composition of the
natural isomorphisms
$$a_L^*(Hdg^\bullet(E)_{\C})\cong a_L^*\C_{E}\cong \C_{E_L}\cong
\Omega^\bullet_{E_L}$$ in $D^+(E_L,\C)$.

In view of Example \ref{pieces}, it suffices to prove the
following claim: \textit{there exists a commutative diagram
$$\begin{CD}
a_L^*(Hdg^\bullet(E)_{\C})@>\beta_L>> a_L^*\psi^H_f(\C)
\\ @V\gamma_L VV @VV\delta_L V
\\\Omega^\bullet_{E_L}@>\sigma_L>> a_L^*C_L^\bullet
\end{CD}$$ in $D^+(E_L,\C)$, where the vertical arrows are isomorphisms, $\sigma_L$ is the map in (\ref{s_L}), and
$\delta_L$ is a morphism in $C^+(E_L,\C)$ such that the
composition
$$\begin{CD}
v_J^*R\psi_{f}(\C)\cong
s_{\square}(((b_L)_*a_L^*\psi_f^H(\C))_{L\in \square_{J}})
@>s_{\square}(\delta_L)>> s_{\square}((C_L^{\bullet})_{L\in
\square_J})= v_J^*\psi^H_f(\C)\end{CD}$$ coincides with the
natural comparison isomorphism $v_J^*R\psi_{f}(\C)\cong
v_J^*\psi^H_f(\C)$ in $D^+(E(J),\C)$.}

Consider the natural map
$$\delta'_L:a_L^*(\Omega^{\bullet}(\log E)/W_0\Omega^{\bullet}(\log E))\rightarrow a_L^*(\Omega^{\bullet}(\log
E)|_{E_J}/W_0\Omega^{\bullet}(\log E)|_{E_J})$$ in $C^+(E_L,\C)$.
Applying the functor $a_L^*$ to the diagram in Lemma \ref{res}, we
see immediately that $\delta'_L$ is a filtered quasi-isomorphism
w.r.t. the quotient of the weight filtrations. The map $\delta_L'$
defines a morphism of double complexes $a_L^*A^{p,q}\rightarrow
a_L^*A^{p,q}_L$, inducing a map $\delta_L$ on the associated
simple complexes. Since $\delta'_L$ is a filtered
quasi-isomorphism, $\delta_L$ is a filtered quasi-isomorphism
w.r.t. the second filtration on the double complexes.
%By exactness of $a_L^*:C^+(E,\C)\rightarrow C^+(E_L,\C)$, we have
%the following chain of isomorphisms in $D^+(E_L,\C)$, for each
%integer $m\geq 1$ (the second one comes from the residue map).
%\begin{eqnarray*}
%Gr^W_m(a_L^*\Omega^\bullet_X(\log E))&\cong& a_L^* Gr^W_m
%\Omega^\bullet_X(\log E)
%\\ &\cong& a_L^*(\bigoplus_{\stackrel{M\subset
%I}{|M|=m}}(a_M)_*\Omega^{\bullet-m}_{E_M})\cong
%a_L^*(\bigoplus_{\stackrel{M\subset I}{|M|=m}}(a_M)_*\C_{E_M}[-m])
%\end{eqnarray*}
% On the other hand, by Lemma \ref{res},
%$$a_L^*Gr^W_m(\Omega^\bullet_{X}(\log E)|_{E_L})\cong a_L^*(\bigoplus_{\stackrel{M\subset
%I}{|M|=m}}(a_{L\cup M})_*\Omega^{\bullet-m}_{E_{L\cup M}})\cong
%a_L^*(\bigoplus_{\stackrel{M\subset I}{|M|=m}}(a_{L\cup
%M})_*\C_{E_{L\cup M}}[-m])$$ in $D^+(E_L,\C)$,
%%where we denote by
%%$a_{L\cup M}^M$ the inclusion $E_{L\cup M}\rightarrow E_L$,
%and
%the diagram
%$$\begin{CD}
%Gr^W_m(a_L^*\Omega^\bullet_X(\log E))@>\cong >>
%a_L^*(\bigoplus_{\stackrel{M\subset I}{|M|=m}}(a_M)_*\C_{E_M}[-k])
%\\@VGr^W_m(\delta') VV @VV\cong V
%\\a_L^*Gr^W_m(\Omega^\bullet_{X}(\log E)|_{E_L})@>\cong >> a_L^*(\bigoplus_{\stackrel{M\subset I}{|M|=m}}(a_{L\cup M})_*
%\C_{E_{L\cup M}}[-k])
%\end{CD}$$
%commutes.
The fact that $\sigma_L\circ \gamma_L=\delta_L\circ \beta_L$
follows from Lemma \ref{res}.
\end{proof}
\begin{prop}\label{mhs}
The natural specialization map $\Z\rightarrow R\psi_{f}(\Z)$
 induces canonical
isomorphisms
$$Gr_F^0H^i(E(J),\C)\cong Gr_F^0\mathbb{H}^i(E(J),v_J^*R\psi_f(\C))$$
for $i\geq 0$. In particular, by restricting to the weight zero
part, we get canonical isomorphisms
$$W_0H^i(E(J),\Q)\cong W_0\mathbb{H}^i(E(J),v_J^*R\psi_f(\Q))$$
\end{prop}
\begin{proof}
It suffices to prove the statement after replacing rational
coefficients by complex coefficients. Then the second isomorphism
follows from the first since both sides have Hodge type $(0,0)$.

Applying the functor $Gr_F^0$ to the morphism $\sigma_L$ in
(\ref{s_L}), we get a morphism of complexes
$$Gr_F^0(\sigma_L):(a_{L})_*\mathcal{O}_{E_L}\rightarrow
\Omega_X^{\bullet+1}(\log
E)|_{E_L}/W_\bullet(\Omega_X^{\bullet+1}(\log
E)|_{E_L})=Gr^W_{\bullet+1}(\Omega_X^{\bullet+1}(\log E)|_{E_L})$$
and this map is a quasi-isomorphism: by Lemma \ref{res}, the
target is isomorphic to the complex $\oplus_{M\subset
I,|M|=\bullet+1}(a_{L\cup M})_*\mathcal{O}_{E_{L\cup M}}$. Passing
to the associated simple complexes, we see that
$$Gr_F^0(\sigma):Gr^0_F Hdg^\bullet(E(J))_{\C}\rightarrow Gr_F^0\psi^H_f(\C)$$ is a quasi-isomorphism, as
well.
\end{proof}

\subsection{Mixed Hodge structure on the nearby and vanishing cohomology}
Consider a complex variety $X$ and a flat morphism $f:X\rightarrow
\mathrm{Spec}\,\C[t]$ with smooth generic fiber. Let $x$ be any
complex point of $X_s$. The cohomology spaces $R^i\psi_f(\Z)_x$
carry a natural mixed Hodge structure
\cite[15.13]{navarro-invent}\cite[\S\,12.1.2]{peters-steenbrink}.
Take a strictly semi-stable model $(Y,g,\varphi)$ of $(f,x)$
(Definition \ref{red}), with $Y_s=\sum_{i\in I}E_i$; then
$\varphi^{-1}(x)=E(J)$ for some non-empty subset $J$ of $I$. There
are canonical isomorphisms
$$\mathbb{H}^i(E(J),v^*_JR\psi_{g}(\Z))\cong R^i\psi_f(\Z)_x $$
for $i\geq 0$, and in this way $R^i\psi_f(\Z)_x$ inherits an
integral mixed Hodge structure, which does not depend on the
chosen strictly semi-stable model.

If we denote by $R\Theta_f$ the vanishing cycles functor, then
$R^i\Theta_f(\Z)_x$ also carries a natural mixed Hodge structure
\cite[1.1]{navarro-survey}. We have
$$R^i\Theta_f(\Z)_x\cong R^i\psi_f(\Z)_x$$ for $i>0$, so
$R^i\Theta_f(\Z)_x$ inherits a mixed Hodge structure. For $i=0$,
 $R^i\Theta_f(\Z)_x$ carries a pure Hodge structure of
weight zero.

If $X$ is smooth at $x$, then $R^i\psi_f(\Z)_x$ is isomorphic to
the degree $i$ integral singular cohomology of the topological
Milnor fiber of $f$ at $x$ for each $i$, so this cohomology
carries a mixed Hodge structure. If $x$ is an isolated singularity
of $f$, then this mixed Hodge structure was constructed in
\cite{steenbrink-vanish}. Likewise, $R^i\Theta_f(\Z)_x$ is
isomorphic to the degree $i$ reduced integral singular cohomology
of the topological Milnor fiber of $f$ at $x$.
\subsection{Non-archimedean interpretation of the weight zero subspace}
In this section, we put $R=\C[[t]]$ and $K=\C((t))$, and we endow
$K$ with the absolute value $|\cdot|_r$ for some fixed $r\in \,
]0,1[$.
\begin{theorem}\label{main}
Let $X$ be a complex variety, endowed with a flat morphism
$f:X\rightarrow \mathrm{Spec}\,\C[t]$ with smooth generic fiber.
Let $x$ be a point of $X(\C)$ with $f(x)=0$. Denote by
$\mathscr{F}_x$ the analytic Milnor fiber of $f$ at $x$.

%(1)
For each $i\geq 0$, there exists canonical isomorphisms
\begin{eqnarray*}
\alpha&:&H^i(\overline{\mathscr{F}}_{\!\! x},\Q)\rightarrow
W_0R^i\psi_f(\Q)_x
\\ \alpha'&:& \widetilde{H}^i(\overline{\mathscr{F}}_{\!\! x},\Q)\rightarrow
W_0R^i\Theta_f(\Q)_x
\end{eqnarray*}
%which is an isomorphism for $i=0$%.
%
%(2) If $X$ is smooth,
%%and $f$ has an isolated singularity at $x$,
%then
%whose images are precisely the weight zero subspace
%$W_0R^i\psi_f(\Q)_x$ (resp. $W_0R^i\Theta_f(\Q)_x$)  of the mixed
%Hodge structure on $R^i\psi_f(\Q)_x$ (resp. $R^i\Theta_f(\Q)_x$).
%In particular, $\alpha$ and $\alpha'$ are isomorphisms for $i=0$.
\end{theorem}
\begin{proof}
Take a strictly semi-stable model $(Y,g,\varphi)$ of $(f,x)$
(Definition \ref{red}), denote by $d$ its ramification index, and
by $\mY$ the $t$-adic completion of $Y$. Put $E=\varphi^{-1}(x)$,
and denote by
 $v:E\rightarrow \mY_s$ the closed immersion.
  Then $\varphi$
induces an isomorphism of $\widehat{(K_r)^a}$-analytic spaces
$$\overline{]E[}:=sp_{\mY(r)}^{-1}(E)\widehat{\times}_{K_r(d)}\widehat{(K_r)^a}\cong
\overline{\mathscr{F}}_{\!\! x}$$ and an isomorphism of mixed
Hodge structures
$$R^i\psi_f(\Q)_x\cong \mathbb{H}^i(E,v^*R\psi_{g}(\Q))$$ On the other hand, Berkovich proved in
\cite[Thm\,1.1(c)]{berk-tate} that there exists a canonical
isomorphism $H^i(E^{an},\Q)\cong W_0H^i(E,\Q)$, and it follows
from Proposition \ref{homeq} and Lemma \ref{extension}
 that there exists a canonical isomorphism
$H^i(E^{an},\Q)\cong H^i(\,\overline{]E[}\,,\Q)$.

Hence, we obtain an isomorphism $H^i(\overline{\mathscr{F}}_{\!\!
x},\Q)\cong W_0H^i(E,\Q)$ and, by Proposition \ref{mhs}, an
isomorphism
$$H^i(\overline{\mathscr{F}}_{\!\! x},\Q)\cong W_0R^i\psi_f(\Q)_x$$
A standard argument shows that this isomorphism does not depend on
the chosen semi-stable model.

The mixed Hodge structure on $R^0\psi_f(\Q)_x$ is pure of weight
zero, so for $i=0$ the map $\alpha$ is an isomorphism
$$\alpha:H^0(\overline{\mathscr{F}}_{\!\! x},\Q)\rightarrow
R^0\psi_f(\Q)_x$$ Passing to reduced cohomology yields the natural
isomorphism $\alpha'$.
\end{proof}
\begin{remark} Using the same methods as in \cite{berk-limit} one can
generalize Theorem \ref{main} as follows: if $f:X\rightarrow \Spec
\C[t]$ is a flat morphism of complex varieties, $Z$ a proper
subvariety of the special fiber, and $\mZ$ the formal completion
of $f$ along $Z$, then there is for each $i\geq 0$ a canonical
isomorphism
\begin{eqnarray*}
\alpha&:&H^i(\overline{\mZ(r)_\eta},\Q)\rightarrow
W_0\mathbb{H}^i(Z,R\psi_f(\Q))
%\\ \alpha'&:& H^i(\overline{\mZ_\eta},\Q)\rightarrow
%W_0\mathbb{H}^i(Z,R\Theta_f(\Q))
\end{eqnarray*}
One reduces to the case where $f$ has smooth generic fiber by
taking a hypercovering of $f$; then one uses the proof of Theorem
\ref{main}.
\end{remark}
\begin{corollary}\label{vanishcoh}
If $X$ is smooth, of pure dimension $n+1$,  and if we denote by
$s$ the dimension of the singular locus of $f$ at $x$, then
$H^i(\overline{\mathscr{F}}_{\!\! x},\Q)=0$ for $i\notin
\{0,n-s,n-s+1,\ldots,n\}$. If $s<n$, then
$\overline{\mathscr{F}}_{\!\! x}$ is arc-connected.
\end{corollary}
\begin{proof}
It is well-known that the reduced cohomology of the topological
Milnor fiber of $f$ at $x$ vanishes for $i\notin
\{n-s,n-s+1,\ldots,n\}$. If $s<n$, then $R^0\psi_f(\Q)_x\cong \Q$,
so $\overline{\mathscr{F}}_{\!\! x}$ is connected and hence
arc-connected by \cite[3.2.1]{Berk1}.
\end{proof}
%Note that Corollary \ref{vanishcoh} imposes strong restrictions on
%the possible configurations of exceptional components in a
%strictly semi-stable model of $(f,x)$, via the isomorphism
%$H^*(|\Delta(E)|,\Q) \cong H^*(\overline{\mathscr{F}}_{\!\!
%x},\Q)$ with $E$ as in the proof of Theorem \ref{main}.
\begin{cor}\label{steen}
We assume that $X$ is smooth of pure dimension $n+1$ with $n>0$,
$f$ is projective and $x$ is the only singular point of the
special fiber $X_s=f^{-1}(0)$. We denote by $\X$ the $t$-adic
completion of $f$. Then we have an isomorphism of long exact
sequences
$$\minCDarrowwidth5pt\begin{CD}
 @>>> H^i((X_s)^{an},\Q)@>>> H^i(\overline{\X(r)}_\eta,\Q)@>>>
\widetilde{H}^i(\,\overline{\mathscr{F}}_{\!\! x},\Q)@>>>
H^{i+1}((X_s)^{an},\Q)@>>>
\\@.@V\gamma V\cong V @V\beta V\cong V @V\cong V\alpha' V @VV\cong
V@.
\\ @>>> W_0H^i(X_s,\Q)@>>>
W_0\mathbb{H}^i(X_s,R\psi_f(\Q))@>>> W_0R^i\Theta_f(\Q)_x@>>>
W_0H^{i+1}(X_s,\Q)@>>>
\end{CD} $$
where the upper long exact sequence is the one from Proposition
\ref{longexact}, the lower one is obtained by applying the exact
functor $Gr^W_0$ to the canonical long exact sequence of mixed
Hodge structures \cite[1.1]{navarro-survey}
$$\ldots\rightarrow H^i(X_s,\Q)\rightarrow \mathbb{H}^i(X_s,R\psi_f(\Q))\rightarrow
\mathbb{H}^i(X_s,R\Theta_f(\Q))\cong
R^i\Theta_{f}(\Q)_x\rightarrow \ldots$$
 and where $\beta$ is the isomorphism of
\cite[5.1]{berk-limit} and $\gamma$ is the isomorphism of
\cite[Thm\,1.1(c)]{berk-tate}.

This diagram breaks up in commutative squares of isomorphisms, for
each $0\leq i<n$,
$$\begin{CD}
H^i((X_s)^{an},\Q)@>\sim >> H^i(\overline{\X(r)}_\eta,\Q)
\\ @V\gamma V\cong V @V\cong V\beta V
\\ W_0H^i(X_s,\Q)@>\sim >>
W_0\mathbb{H}^i(X_s,R\psi_f(\Q))
\end{CD}$$
as well as an isomorphism of exact sequences
$$\minCDarrowwidth5pt\begin{CD}
0@>>> H^n((X_s)^{an},\Q)@>>> H^n(\overline{\X(r)}_\eta,\Q)@>>>
H^n(\,\overline{\mathscr{F}}_{\!\! x},\Q)@>>>
H^{n+1}((X_s)^{an},\Q)@>>> 0
\\@. @V\gamma V\cong V @V\beta V\cong V @V\cong V\alpha V @VV\cong
V @.
\\0@>>> W_0H^n(X_s,\Q)@>>>
W_0\mathbb{H}^n(X_s,R\psi_f(\Q))@>>> W_0R^n\psi_f(\Q)_x@>>>
W_0H^{n+1}(X_s,\Q)@>>> 0
\end{CD} $$
\end{cor}
It is well-known that any analytic germ of an isolated
hypersurface singularity can be embedded in a projective morphism
$f$ as in Corollary \ref{steen}; see \cite[\S\,1.1]{brieskorn}.
\section{Comparison with the motivic Milnor
fiber}\label{sec-motmil} In this final section, we consider the
motivic counterparts of the above results. Let $k$ be a field of
characteristic zero, and put $R=k[[t]]$ and $K=k((t))$. We fix an
element $r$ in $]0,1[$ and we endow $K$ with the $t$-adic absolute
value $|\cdot|_r$. Since $r$ will remain fixed throughout this
section, we simplify notation by writing $K$ for the
non-archimedean field $K_r$ and $\X_\eta$ for the generic fiber
$\X(r)_\eta$ of a special formal $R$-scheme in the category of
$K_r$-analytic spaces.
 For each integer $d>0$, we put
$K(d)=K[t_d]/((t_d)^d-t)$ and we denote by $R(d)$ the
normalization of $R$ in $K(d)$.
%We endow $K(d)$ with the unique
%extension of the absolute value on $K$. To simplify notation,
%we'll write $\X_\eta$ instead of $\X(r)_\eta$ for the generic
%fiber of a special formal $R$-scheme in the category of
%$K(,|\cdot|_r)$-analytic spaces.

\subsection{Motivic volume of a formal scheme}
Let $\X$ be a generically smooth special formal $R$-scheme. In
\cite[7.39]{Ni-trace} we defined the motivic volume
$S(\X;\widehat{K^s})$ of $\X$. It is an element of the localized
Grothendieck ring of $\X_0$-varieties $\mathcal{M}_{\X_0}$ (see
for instance \cite[\S\,3.1]{NiSe} for the definition of
$\mathcal{M}_{\X_0}$). The aim of the current section is to
explain the behaviour of the motivic volume under extension of the
base ring $R$. This result will be used further on to give an
expression of the motivic volume in terms of a semi-stable model
of $\X$ (Theorem \ref{motmil}).

As we noted in \cite[7.40]{Ni-trace}, the motivic volume depends
in general on the choice of the uniformizer $t$, or, more
precisely, on the $K$-fields $K(d)$. If $k$ is algebraically
closed, then up to $K$-isomorphism, $K(d)$ is the unique extension
of $K$ of degree $d$, and the motivic volume is independent of the
choice of uniformizer.

In order to speak about the motivic volume of a special formal
$R(d)$-scheme when $k$ is not algebraically closed, we fix the
uniformizer $t_d$ in $R(d)$. This yields an isomorphism of
$k$-algebras $R(d)\cong k[[t_d]]$ and natural isomorphisms of
$R$-algebras $(R(d))(e)\cong R(d\cdot e)$ for $d,e>0$.

As usual, we denote by $\LL$ the class of the affine line
$\A^1_{\X_0}$ in $\mathcal{M}_{\X_0}$. This is a unit in
$\mathcal{M}_{\X_0}$.
%Note that, for any integer $d>0$, $(\X\times_R R(d))_0$ is
%canonically isomorphic to $\X_0$.
 We'll denote by $\mathcal{R}_{\X_0}$ the subring
$$\mathcal{R}_{\X_0}:=
\mathcal{M}_{\X_0}\left[\frac{T^b}{T^b-\LL^a} \right]_{(a,b)\in
\Z\times \N^{\ast}}$$ of $\mathcal{M}_{\X_0}[[T]]$, and by
$\mathcal{R}'_{\X_0}$ the subring of $\mathcal{M}_{\X_0}[[T]]$
consisting of elements of the form $P(T)/Q(T)$, with $P(T),\,Q(T)$
polynomials over $\mathcal{M}_{\X_0}$ such that $Q(0)$ is a unit
in $\mathcal{M}_{\X_0}$, $Q(T)$ is monic, and the degree of $Q(T)$
is at least the degree of $P(T)$. The ring $\mathcal{R}'_{\X_0}$
contains $\mathcal{R}_{\X_0}$. There exists a unique morphism of
$\mathcal{M}_{\X_0}$-algebras
$$\lim_{T\to \infty}:\mathcal{R}'_{\X_0}\rightarrow \mathcal{M}_{\X_0}$$
mapping $P(T)/Q(T)$ (with $P(T),\,Q(T)$ as above) to the
coefficient of $T^{\mathrm{deg}\,Q}$ in $P(T)$, where
$\mathrm{deg}\,Q$ denotes the degree of $Q(T)$. It restricts to
the morphism
$$\lim_{T\to \infty}:\mathcal{R}_{\X_0}\rightarrow \mathcal{M}_{\X_0}$$
from \cite[7.35]{Ni-trace}.

%$$\lim_{T\to \infty}:\mathcal{R}_{\X_0}\rightarrow \mathcal{M}_{\X_0}$$
%the unique morphism of $\mathcal{M}_{\X_0}$-algebras mapping
%$\LL^aT^b/(1-\LL^aT^b)$ to $-1$ for every $(a,b)\in \Z\times
%\N^{\ast}$.

For any ring $A$, any element $a(T)=\sum_{i\geq 0}a_iT^i$ of
$A[[T]]$, and any integer $d>0$, we put $a(T)[d]=\sum_{i\geq
0,\,d|i}a_{i}T^{i}\in A[[T]]$.

\begin{lemma}\label{R}
For any $(p,q,r)\in \Z\times \N\times \N^{\ast}$ with $q\leq r$,
$$\frac{T^q}{T^r-\LL^p}\in \mathcal{R}_{\X_0}$$
\end{lemma}
\begin{proof}
For any $(p,q,r)\in \Z\times \N\times \N^{\ast}$ we use the
notation
$$D^{p}_{q,r}:=\frac{T^q}{T^r-\LL^p}\ \in \mathcal{M}_{\X_0}[[T]]$$
By definition, $D^{p}_{r,r}\in \mathcal{R}_{\X_0}$. Moreover, the
relation $D^p_{0,r}=\LL^{-p}\left(D^{p}_{r,r}-1\right)$ shows that
$D^{p}_{0,r}\in \mathcal{R}_{\X_0}$.

Put $I=\{(q,r)\in \N\times \N^{\ast}\,|\,q\leq r\}$. We will show
that $D^{p}_{q,r}\in \mathcal{R}_{\X_0}$ for any $(p,q,r)\in
\Z\times I$, by induction on $(q,r)$ w.r.t. the lexicographic
ordering on $I$. We may assume that $0<q<r$. We have
\begin{eqnarray*} D^{p}_{q,r}&=&\frac{T^r}{T^r-\LL^p}\cdot
\frac{1}{T^{r-q}-1}-\frac{1}{T^r-\LL^p}\cdot\frac{T^q}{T^{r-q}-1}
\\&\equiv& -D^{p}_{0,r}\cdot D^{0}_{q,r-q}\mod \mathcal{R}_{\X_0}
\end{eqnarray*}
We know that $D^{p}_{0,r}\in \mathcal{R}_{\X_0}$, and if $q\leq
r-q$, then $D^{0}_{q,r-q}\in \mathcal{R}_{\X_0}$ by the induction
hypothesis. Hence, we may assume that $q>r-q$. Then we write
$$D^{p}_{0,r}\cdot D^{0}_{q,r-q}=D^{p}_{2q-r,r}\cdot D^{0}_{r-q,r-q}$$
and since $0<2q-r<q$ the induction hypothesis implies that
$D^{p}_{2q-r,r}$ belongs to $\mathcal{R}_{\X_0}$.
\end{proof}
\begin{lemma}\label{divis}
For any integer $d>0$, the morphism of
$\mathcal{M}_{\X_0}$-modules
$$\phi_d:\mathcal{M}_{\X_0}[[T]]\rightarrow
\mathcal{M}_{\X_0}[[T]]:a(T)\mapsto a(T)[d]$$ restricts to a
$\mathcal{M}_{\X_0}$-module endomorphism of $\mathcal{R}_{\X_0}$.
 Moreover, $$\lim_{T\to \infty}\circ\ \phi_d=\lim_{T\to \infty}$$
\end{lemma}
\begin{proof}
Consider an integer $n>0$, a tuple $a\in \Z^{n}$, and a tuple
$b\in (\N^{\ast})^n$. We put
$$C_{a,b}:=\left(\prod_{i=1}^{n}\frac{T^{b_i}}{T^{b_i}-\LL^{a_i}}\right)[d]\in
\mathcal{M}_{\X_0}[[T]]$$ It suffices to show that $C_{a,b}\in
\mathcal{R}_{\X_0}$ and that $$\lim_{T\to\infty}C_{a,b}=1$$

For any pair of tuples $u,v\in \Z^n$, we put $u\cdot
v=\sum_{i=1}^{n}u_i\cdot v_i$. For each $i\in \{1,\ldots,n\}$, we
put $m_i=lcm(b_i,d)$ and $e_i=m_i/b_i$. If we put
$$S=\{u\in \N^n\,|\,1\leq u_i\leq e_i\mbox{ for all }i,\mbox{ and }d|u\cdot b \}$$
then the map
$$S\times \N^{n}\rightarrow \{w\in (\N^{\ast})^n\,|\,w\cdot b\in d\N\}:(u,v)\mapsto (u_1+v_1e_1,\ldots,u_n+v_ne_n) $$
is a bijection. Therefore
$$C_{a,b}=\left(\sum_{u\in S}\LL^{-a\cdot u}T^{b\cdot u}\right)\cdot\left(\prod_{i=1}^n \frac{\LL^{a_ie_i}}{T^{b_ie_i}-\LL^{a_ie_i}} \right)$$
Since $e=(e_1,\ldots,e_n)$ belongs to $S$, the first factor is a
polynomial with leading term $\LL^{-a\cdot e}T^{b\cdot e}$, so we
see that $C_{a,b}$ belongs to $\mathcal{R}'_{\X_0}$ and
$$\lim_{T\to\infty}C_{a,b}=1$$ as required. Lemma \ref{R} shows
that $C_{a,b}\in \mathcal{R}_{\X_0}$.
\end{proof}

\begin{prop}\label{bchange}
 For any generically smooth special formal $R$-scheme $\X$ and any integer $d>0$,
$$S(\X;\widehat{K^s})=S(\X\times_R R(d);\widehat{K(d)^s})$$ in
$\mathcal{M}_{\X_0}$.
\end{prop}
\begin{proof}
By \cite[7.38+39]{Ni-trace} we may assume that $\X_\eta$ admits a
$\X$-bounded gauge form $\omega$ (in the sense of
\cite[2.11]{Ni-trace}). For each $n>0$ we denote by $\omega(n)$
the pull-back of $\omega$ to $\X_\eta\times_K K(n)$. This is a
$\X\times_R R(n)$-bounded gauge form.

 In \cite[4.9]{Ni-trace} we defined the
volume Poincar\'e series $S(\X,\omega;T)$ of $(\X,\omega)$ by
$$S(\X,\omega;T)=\sum_{n>0}\left(\int_{\X\times_R R(n)}|\omega(n)|\right)T^n\quad \in \mathcal{M}_{\X_0}[[T]] $$
(the coefficients are motivic integrals). It is clear from the
definition (and from our choice of uniformizer in $R(d)$) that
$$S(\X\times_R R(d),\omega(d);T)=S(\X,\omega;T)[d]$$
We showed in \cite[7.14]{Ni-trace} that $S(\X,\omega;T)$ is
rational; more precisely, it belongs to $\mathcal{R}_{\X_0}$.
%We denote by $\mathcal{R}$ the
%subring of $\mathcal{M}_{\X_0}[[T]]$ consisting of elements of the
%form $P(T)/Q(T)$, with $P(T),\,Q(T)$ polynomials over
%$\mathcal{M}_{\X_0}$ such that $Q(0)$ is a unit in
%$\mathcal{M}_{\X_0}$, $Q(T)$ is monic, and the degree of $Q(T)$ is
%at least the degree of $P(T)$. The ring $\mathcal{R}$ contains
%$\mathcal{R}'$. There exists a unique morphism of
%$\mathcal{M}_{\X_0}$-algebras
%$$\lim_{T\to \infty}:\mathcal{R}\rightarrow \mathcal{M}_{\X_0}$$
%mapping $P(T)/Q(T)$ (with $P(T),\,Q(T)$ as above) to the
%coefficient of $T^{\mathrm{deg}\,Q}$ in $P(T)$, with
%$\mathrm{deg}\,Q$ the degree of $Q(T)$. It restricts to the
%morphism
%$$\lim_{T\to \infty}\mathcal{R}'\rightarrow \mathcal{M}_{\X_0}$$
%from \cite[7.33]{Ni-trace}.
 By definition,
$$S(\X;\widehat{K^s})=-\lim_{T\to \infty}S(\X,\omega;T)$$
%It is not hard to see that the $\mathcal{M}_{\X_0}$-module
%endomorphism of $\mathcal{M}_{\X_0}[[T]]$ which maps a power
%series $a(T)$ to $a(T)[d]$ restricts to a
%$\mathcal{M}_{\X_0}$-module endomorphism $\phi$ of $\mathcal{R}$,
%and
%$$\lim_{T\to \infty}\circ \ \phi= \lim_{T\to \infty}$$
 By Lemma \ref{divis},
 \begin{eqnarray*}
 S(\X;\widehat{K^s})&=&-\lim_{T\to
\infty}S(\X,\omega;T)\\&=&-\lim_{T\to \infty}S(\X\times_R
R(d),\omega(d);T) \\&=&S(\X\times_R R(d);\widehat{K(d)^s})
\end{eqnarray*}
\end{proof}
%\begin{cor}\label{bchange2}
%Consider a generically smooth special formal $R$-scheme $\X$ and a
%finite extension $L$ of $K$ in $K^s$. We denote by $R_L$ the
%normalization of $R$ in $L$, and we choose a uniformizer $t_L$ in
%$R_L$ such that $(t_L)^e=t$ for some $e>0$. If we put
%$\X_L=\X\times_RR_L$, then the motivic volume
%$S(\X_L;\widehat{L^s})$ of $\X_L$ (w.r.t. the uniformizer $t_L$)
%equals the image of $S(\X;\widehat{K^s})$ under the base change
%morphism $\mathcal{M}_{\X_0}\rightarrow \mathcal{M}_{(\X_L)_0}$.
%\end{cor}
%\begin{proof}
%If $L/K$ is unramified this follows easily from the definition of
%the motivic volume. The general case follows by decomposing $K'/K$
%into an unramified and a totally ramified extension.
%\end{proof}
\begin{prop}\label{bchange-unr}
Let $k'$ be a field containing $k$, and put $R'=k'[[t]]$. Let $\X$
be a generically smooth special formal $R$-scheme, and put
$\X'=\X\widehat{\times}_R R'$. Then $S(\X';\widehat{K^s})$ is the
image of $S(\X;\widehat{K^s})$ under the base change morphism
$\mathcal{M}_{\X_0}\rightarrow \mathcal{M}_{\X'_0}$.
\end{prop}
\begin{proof}
This is obvious from the definition of the motivic volume.
\end{proof}

\subsection{Motivic volume of a rigid variety}
Let $\X$ be a generically smooth special formal $R$-scheme, and
assume that $\X$ is $stft$ or that $\X_\eta$ admits a universally
bounded gauge form in the sense of \cite[7.41]{Ni-trace}. Such a
universally bounded gauge form exists, in particular, if $\X$ is
the formal spectrum of a regular local $R$-algebra
\cite[7.23+42]{Ni-trace}.

 In
\cite[8.3]{NiSe} and \cite[7.43]{Ni-trace} we defined the motivic
volume $S(\X_\eta;\widehat{K^s})$ of $\X_\eta$ as the image of
$S(\X;\widehat{K^s})$ under the forgetful morphism
$\mathcal{M}_{\X_0}\rightarrow \mathcal{M}_k$ and we showed that
this definition only depends on $\X_\eta$ and not on the model
$\X$.

\begin{prop}\label{bchange-rig}
Let $d>0$ be an integer and
 $X$ a separated smooth
rigid $K$-variety. Assume that $X$ and $X\times_K K(d)$ admit
universally bounded gauge forms, or that $X$ is quasi-compact.
Then $$S(X\times_K K(d);\widehat{K(d)^s})=S(X;\widehat{K^s})$$ in
$\mathcal{M}_k$.
\end{prop}
\begin{proof}
This follows immediately from Proposition \ref{bchange}.
\end{proof}
\begin{prop}\label{bchange-unr-rig}
Let $k'$ be any field containing $k$, and put $L=k'((t))$. Let $X$
be a separated smooth rigid $K$-variety. Assume that $X$ and
$X\times_K L$ admit universally bounded gauge forms, or that $X$
is quasi-compact. Then $S(X\times_K L;\widehat{L^s})$ is the image
of $S(X;\widehat{K^s})$ under the base change morphism
$\mathcal{M}_k\rightarrow \mathcal{M}_{k'}$.
\end{prop}
\begin{proof}
This follows immediately from Proposition \ref{bchange-unr}.
\end{proof}

\begin{remark}
Even though the valuation on $\widehat{K^a}$ is not discrete and
the construction of the motivic volume does not apply to
$\widehat{K^a}$-analytic spaces, the above results justify the
hope that one can associate a motivic volume $Vol(X)$ to separated
smooth quasi-compact rigid varieties $X$ over $\widehat{K^a}$ (or
an even larger class of $\widehat{K^a}$-analytic spaces). This
volume should have the property that $Vol(Y\widehat{\times}_K
\widehat{K^a})=S(Y;\widehat{K^s})$ when $Y$ is a separated smooth
quasi-compact rigid $K$-variety, and opens the way to a theory of
motivic integration on $\widehat{K^a}$-analytic spaces.
\end{remark}
\subsection{Motivic Milnor fiber}\label{sec-mmil}
We start with an auxiliary result.
\begin{lemma}\label{auxil}
For any morphism of generically smooth special formal $R$-schemes
$h:\mY\rightarrow \X$ such that $h_\eta$ is an isomorphism, we
have $$S(\mY;\widehat{K^s})=S(\X;\widehat{K^s})$$ in
$\mathcal{M}_{\X_0}$. Here the left hand side is viewed as an
element of $\mathcal{M}_{\X_0}$ via the forgetful morphism
$\mathcal{M}_{\mY_0}\rightarrow \mathcal{M}_{\X_0}$ induced by
$h_0$.
\end{lemma}
\begin{proof}
This is straightforward from the definition of the motivic volume
in \cite[7.39]{Ni-trace}.
\end{proof}

 Let $X$
be a $k$-variety of pure dimension $m$, and let $f:X\rightarrow
\A^1_k=\Spec k[t]$ be a flat morphism with smooth generic fiber.
We fix a closed point $x$ on the special fiber $X_s$ of $f$. We
denote by $\X$ the $t$-adic completion of $f$, and by $\X_x$ the
completion of $f$ at $x$, i.e. the special formal $R$-scheme $\Spf
\widehat{\mathcal{O}}_{X,x}$. Then $(\X_x)_\eta=\mathscr{F}_x$ is
the analytic Milnor fiber of $f$ at $x$. We view $\mathscr{F}_x$
as a $k'((t))$-analytic space, with $k'$ the residue field of $x$.

%\begin{lemma}\label{univbound}
%The analytic Milnor fiber $\mathscr{F}_x$ admits a universally
%bounded gauge form.
%\end{lemma}
%\begin{proof}
%Since $\X_x$ is affine and $\mathscr{F}_x$ reduced, any
%$\X_x$-bounded gauge form on $\mathscr{F}_x$ is universally
%bounded, by \cite[7.39]{Ni-trace}, so it suffices to construct a
%\X_x$-bounded gauge form on $\mathscr{F}_x$.
%\end{proof}
If $X$ is smooth at $x$ then $\X_x$ is the formal spectrum of a
regular local $R$-algebra, so $S(\mathscr{F}_x;\widehat{K^s})$ is
defined and equal to $S(\X_x;\widehat{K^s})\in \mathcal{M}_x$. If
$X$ is not smooth at $x$, it is not clear if $\mathscr{F}_x$
admits a universally bounded gauge form. However, we still have
the following property.

\begin{prop}
The analytic Milnor fiber $\mathscr{F}_x$, viewed as a
$k'((t))$-analytic space, uniquely determines the motivic volume
$$S(\X_x;\widehat{K^s})\in \mathcal{M}_x$$
\end{prop}
\begin{proof}
The normalization morphism $\widetilde{\X}_x\rightarrow \X_x$ is a
morphism of special formal $R$-schemes which induces an
isomorphism between the generic fibers because $\mathscr{F}_x$ is
normal (even smooth) and normalization commutes with taking
generic fibers \cite[2.1.3]{conrad}. By Lemma \ref{auxil},
$$S(\widetilde{\X}_x;\widehat{K^s})=S(\X_x;\widehat{K^s})\ \in \mathcal{M}_x$$
Moreover, the $k'[[t]]$-algebra of power-bounded analytic
functions on $\mathscr{F}_x$ is canonically isomorphic to
$\mathcal{O}(\widetilde{\X}_x)$ by \cite[7.4.1]{dj-formal}, so
$\mathscr{F}_x$ determines $S(\X_x;\widehat{K^s})$.
\end{proof}

\begin{definition}\label{motnearby}
If $\mY$ is a generically smooth special formal $R$-scheme of pure
relative dimension $d$, then we define the motivic nearby cycles
$\mathcal{S}_{\mY}$ of $\mY$ by
$$\mathcal{S}_{\mY}=\LL^{d}\cdot S(\mY;\widehat{K^s})\in
\mathcal{M}_{\mY_0}$$ Let $X$ be a $k$-variety of pure dimension,
$f:X\rightarrow \Spec k[t]$ a flat morphism with smooth generic
fiber, and $x$ a closed point of the special fiber $X_s$ of $f$.
Denote by $\X/R$ the $t$-adic completion of $f$ and by $\X_x/R$
the formal completion of $f$ at $x$. We define the motivic Milnor
fiber $\mathcal{S}_{f,x}$ of $f$ at $x$ by
$$\mathcal{S}_{f,x}=\mathcal{S}_{\X_x}\in \mathcal{M}_x$$
 and the motivic nearby cycles
$\mathcal{S}_f$ of $f$ by $$\mathcal{S}_f=\mathcal{S}_{\X}\in
\mathcal{M}_{\X_0}$$
\end{definition}

If $X$ is smooth, then Denef and Loeser defined the motivic nearby
cycles of $f$ and the motivic Milnor fiber of $f$ at $x$ in
\cite[\S\,3.5]{DL3}. By \cite[9.8]{Ni-trace} these objects
coincide with the ones introduced in Definition \ref{motnearby}.
Denef and Loeser showed that if $k=\C$ and $X$ is smooth,
$\mathcal{S}_f$ and $\mathcal{S}_{f,x}$ have the same Hodge
realization as the cohomological nearby cycles, resp. Milnor
fiber, in an appropriate Grothendieck ring of mixed Hodge modules
with monodromy action; see \cite[4.2.1]{DL5} and
\cite[3.5.5]{DL3}.

\subsection{Expression in terms of a semi-stable model}
Now we come to the main result of this section: an expression for
the motivic volume of a formal scheme in terms of a semi-stable
model. We'll freely use the notation and terminology from
\cite[\S\,2.5]{Ni-trace} concerning strict normal crossing
divisors on formal schemes.
\begin{definition}
Let $\mU$ be a generically smooth special formal $R$-scheme. A
strictly semi-stable model $(\mV,g)$ for $\mU$ consists of the
following data: \begin{enumerate} \item an integer $d>0$, which we
call the ramification index of the model $(\mV,g)$, \item a
regular special formal $R(d)$-scheme $\mV$, such that $\mV_s$ is a
reduced strict normal crossing divisor on $\mV$, \item a morphism
of special formal $R$-schemes $g:\mV\rightarrow \mU$ which induces
an isomorphism of $K(d)$-analytic spaces $\mV_\eta\cong
\mU_\eta\times_K K(d)$.
\end{enumerate}
%We call $d$ the ramification index of the model $(\mV,g)$.
\end{definition}
Such a strictly semi-stable model exists, in particular, when
$\mY$ is the formal completion of a generically smooth $stft$
formal $R$-scheme along a closed subscheme of its special fiber
(see the remark following Theorem \ref{hom-mil}).

\begin{theorem}\label{motmil}
If $\mU$ is a generically smooth special formal $R$-scheme and
$(\mV,g)$ is a strictly semi-stable model for $\mU$, with
$\mV_s=\sum_{i\in I}\mE_i$, then
\begin{equation}\label{globss}
\mathcal{S}_{\mU}=\sum_{\emptyset \neq J\subset
I}(1-\LL)^{|J|-1}[E_J^o]\ \in \mathcal{M}_{\mU_0}\end{equation} In
particular, the right hand side does not depend on the chosen
strictly semi-stable model.
\end{theorem}
\begin{proof}
 By Proposition \ref{bchange}, we may assume that the
ramification index $d$ of the strictly semi-stable model is equal
to one, i.e. that $g_\eta:\mV_\eta\rightarrow \mU_\eta$ is an
isomorphism. Then by Lemma \ref{auxil} we have
$\mathcal{S}_{\mU}=\mathcal{S}_{\mV}$ in $\mathcal{M}_{\mU_0}$,
and it follows from \cite[7.36]{Ni-trace} that
\begin{equation*}\label{express}
\mathcal{S}_{\mV}=\sum_{\emptyset \neq J\subset
I}(1-\LL)^{|J|-1}[E_J^o]\ \in \mathcal{M}_{\mV_0}\end{equation*}
%Since $\mE$ is normal and $R$-flat, we know by
%\cite[7.4.1]{dj-formal} that the $R$-algebra
%$\mathcal{O}_{\mE}(\mE)$ is isomorphic to the $R$-algebra of
%power-bounded analytic functions on $\mE_\eta\cong \mathscr{F}_x$.
%Again applying \cite[7.4.1]{dj-formal}, we get an isomorphism of
%$R$-algebras
%$$\mathcal{O}_{\widetilde{\X}_x}(\widetilde{\X}_x)\rightarrow
%\mathcal{O}_{\mE}(\mE)$$ which is continuous. Since
%$\widetilde{\X}_x$ is affine, this isomorphism gives rise to a
%morphism of special formal $R$-schemes $g:\mE\rightarrow
%\widetilde{\X}_x$ which is an isomorphism on the generic fibers.
\end{proof}
\begin{remark}
If the formal $R$-scheme $\mV$ in the statement of Theorem
\ref{motmil} is $stft$, then $E_J^o$ coincides with the stratum
$(\mV_s)^o_J$ of the strictly semi-stable $k$-scheme $\mV_s$
(Section \ref{subsec-strictsemi}) for any $\emptyset \neq J\subset
I=Irr(\mV_s)$.
\end{remark}
\begin{cor}\label{cor-motmil}
Consider $f:X\rightarrow \Spec k[t]$, $x$ and $\X_x$ as in
Definition \ref{motnearby}. If $(\mV,g)$ is a strictly semi-stable
model for $\X_x$, with $\mV_s=\sum_{i\in I}\mE_i$, then
\begin{equation}\label{locss}
\mathcal{S}_{f,x}=\sum_{\emptyset\neq J\subset
I}(1-\LL)^{|J|-1}[E_J^o] \ \in \mathcal{M}_{x}\end{equation} In
particular, the right hand side does not depend on the chosen
strictly semi-stable model.
\end{cor}
Note in particular that any strictly semi-stable model
$(Y,g,\varphi)$ of the germ $(f,x)$ (in the sense of Definition
\ref{red}) induces a strictly semi-stable model for $\X_x$ by
taking the formal completion of $Y$ along $\varphi^{-1}(x)$ (of
course, the projectivity condition can be omitted in Definition
\ref{red}). Therefore, (\ref{locss}) also gives an expression for
$\mathcal{S}_{f,x}$ in terms of a semi-stable model of $(f,x)$.

 As a special case, if $X$ is smooth, we get an
expression for Denef and Loeser's motivic nearby cycles and
motivic Milnor fiber in terms of a strictly semi-stable model of
$f$, resp. of $(f,x)$. This expression is not clear from their
definition \cite[\S\,3.5]{DL3}.

Moreover, combining Theorem \ref{motmil} with the formula in
\cite[7.36]{Ni-trace} we get an expression for the right hand
sides of (\ref{globss}) and (\ref{locss}) in terms of a resolution
of singularities of $\mU$, resp. $\X_x$.

\begin{remark} Theorem \ref{motmil} makes it possible to compare our
notion of motivic nearby cycles and of motivic volume of a rigid
variety with the ones introduced by Ayoub
\cite{ayoub}\cite{ayoub-these}, after specialization to an
appropriate Grothendieck ring of $k$-motives. Details will appear
elsewhere.
\end{remark}

Theorem \ref{motmil} is similar in spirit to Theorem \ref{hom-mil}
and the subsequent remark. However, since motivic integrals take
their values in a Grothendieck ring, all ``non-additive''
information is lost when taking the motivic volume
$S(\X_x;\widehat{K^s})$. As we've seen, non-archimedean geometry
provides a powerful additional tool to prove independence results
of this type.

\bibliographystyle{hplain}
\bibliography{wanbib,wanbib2}
\end{document}